\newif\ifcomments\commentsfalse
\newif\iflong\longtrue

\documentclass[a4paper,english,pdfa]{lipics-v2021}
\iflong \usepackage[bibliography=common]{apxproof}
\else   \usepackage[bibliography=common,appendix=strip]{apxproof}
\fi
\usepackage{mathtools}
\usepackage{xspace}

\usepackage{pgf,tikz,pgfplots}
\usetikzlibrary{cd,babel}		
\tikzset{
    symbol/.style={%
        draw=none,
        every to/.append style={%
            edge node={node [sloped, allow upside down, auto=false]{$#1$}}}
    }
}
\newsavebox{\pullback}          
\sbox\pullback{
  \begin{tikzpicture}
    \draw (0,0) -- (1ex,0ex);
    \draw (1ex,0ex) -- (1ex,1ex);
  \end{tikzpicture}
}

\usepackage[paper, notion, quotation]{knowledge}

\knowledge{notion}
  | behaviorally equivalent
  | behavior equivalence
\knowledge{notion}
  | Codensity lifting
  | Codensity liftings
  | codensity lifting
  | codensity liftings
  | Codensity
  | codensity
  | F^\uparrow
\knowledge{notion}
  | Conditional transition system
  | Conditional transition systems
  | conditional transition system
  | conditional transition systems
  | CTS
  | CTSs
\knowledge{notion}
  | Constant functor
  | Constant functors
  | constant functor
  | constant functors
  | constant
  | A
  | B
\knowledge{notion}
  | Coprojection
  | Coprojections
  | coprojection
  | coprojections
  | \kappa_i
  | \kappa_j
  | \kappa_k
\knowledge{notion}
  | Correspondence
  | Correspondences
  | correspondence
  | correspondences
\knowledge{notion}
  | Coupling
  | Couplings
  | coupling
  | couplings
  | \Omega
\knowledge{notion}
  | Coupling-based lifting
  | Coupling-based liftings
  | coupling-based lifting
  | coupling-based liftings
  | Coupling-based
  | coupling-based
  | F^\downarrow
\knowledge{notion}
  | diagonal functor
  | \Delta
\knowledge{notion}
  | Kantorovich-Rubinstein duality
  | Kantorovich-Rubinstein dualities
  | Duality
  | Dualities
  | duality
  | dualities
\knowledge{notion}
  | Equivalent
  | equivalent
  | Equivalence
  | Equivalences
  | equivalence
  | equivalences
  | \sim
\knowledge{notion}
  | Euclidean pseudometric
  | \de
\knowledge{notion}
  | expectation
  | \E
\knowledge{notion}
  | Finite probability distribution functor
  | finite probability distribution functor
  | finite probability distribution functors
  | probability distribution functor
  | probability distribution functors
  | finite distribution functors
  | finite distribution functor
  | distribution functor
  | distribution
  | distribution functors
  | \Dd
\knowledge{notion}
  | \overline{F}
\knowledge{notion}
  | hom
  | [-,-]
\knowledge{notion}
  | identity functor
  | identity functors
  | identity
  | \Id
\knowledge{notion}
  | inequality
  | \le
  | \ge
\knowledge{notion}
  | Lifting
  | Liftings
  | lifting
  | liftings
\knowledge{notion}
  | Modality
  | Modalities
  | modality
  | modalities
  | \tau
  | \tau_i
  | \tau_j
  | \tau_1
  | \tau_2
  | \tau_a
  | \ev
  | \ev_i
  | \ev_j
  | \ev_1
  | \ev_2
  | \ev_a
  | \ev_F
\knowledge{notion}
  | monotone
  | Monotonicity
  | monotonicity
\knowledge{notion}
  | non-expansive functions
  | non-expansive maps
  | non-expansive
\knowledge{notion}
  | Optimal coupling
  | Optimal couplings
  | optimal coupling
  | optimal couplings
\knowledge{notion}
  | Optimal function
  | Optimal functions
  | optimal function
  | optimal functions
  | optimal non-expansive function
  | function
  | functions
\knowledge{notion}
  | simple polynomial functors
\knowledge{notion}
  | Polynomial functor
  | Polynomial functors
  | polynomial functor
  | polynomial functors
\knowledge{notion}
  | Powerset functor
  | powerset functor
  | powerset functors
  | finite powerset functor
  | finite powerset
  | Powerset
  | powerset
  | \Pp
\knowledge{notion}
  | Predicate
  | Predicates
  | predicate
  | predicates
  | $\Vv$-predicate
  | $\Vv$-predicates
  | \Vv\cate{-Pred}
\knowledge{notion}
  | Pseudometric
  | Pseudometrics
  | pseudometric
  | pseudometrics
  | $\Vv$-pseudometric
  | $\Vv$-pseudometrics
  | $\Pp(\Ll)$-pseudometric
  | $\Pp(\Ll)$-pseudometrics
\knowledge{notion}
  | ($\Vv$-pseudometric) morphism
  | pseudometric morphism
  | $\Vv$-pseudometric morphism
  | $\Vv$-pseudometric morphisms
  | morphism
  | morphisms
  | \Vv\cate{-PMet}
  | \Pp(\Ll)\cate{-PMet}
\knowledge{notion}
  | Quantale
  | Quantales
  | quantale
  | quantales
  | \Vv
  | \otimes
\knowledge{notion}
  | finite couplings
\knowledge{notion}
  | set of maximal elements
  | M_{t_1}
  | M_{t_2}
  | M_t
\knowledge{notion}
  | shortest-distinguishing-word-distance
  | d_\sdw
\knowledge{notion}
  | Well-behaved modality
  | Well-behaved modalities
  | Well-behaved
  | well-behaved

\title{Correspondences between codensity and coupling-based liftings, a practical approach}
\titlerunning{Correspondences between Codensity and Coupling-based Liftings}

\author{Samuel Humeau}{ENS de Lyon, CNRS, LIP, UMR 5668, 69342, Lyon cedex 07, France}{samuel.humeau@ens-lyon.fr}{https://orcid.org/0009-0007-1850-9744}{}
\author{Daniela Petrisan}{CNRS, IRIF, Université Paris Diderot, Paris, France}
{[0000-0001-9712-930X]}{}{}
\author{Jurriaan Rot}{Institute for Computing and Information Sciences, Radboud University, Nijmegen,
the Netherlands}{jurriaan.rot@ru.nl}{https://orcid.org/0000-0002-1404-6232}{}
\authorrunning{S. Humeau, D. Petrisan, and J. Rot}
\Copyright{Samuel Humeau, Daniela Petrisan, and Jurriaan Rot}
\funding{This work was supported by the LABEX MILYON (ANR-10-LABX-0070) of Université de Lyon, within the program ``Investissements d'Avenir'' (ANR-11-IDEX- 0007) operated by the French National Research Agency (ANR), and supported by the NWO grant OCENW.M20.053.}
\acknowledgements{The authors would like to thank Pedro Nora for suggestions and discussions.}
\iflong
  \hideLIPIcs
  \pdfoutput=1
  \relatedversiondetails[cite={hpr:csl25:correspondences}]
  {Published version}
  {https://doi.org/10.4230/LIPIcs.CSL.2025.29}
\else
  \relatedversiondetails[cite={hpr:csl25:correspondences:hal}]
  {Full Version}{https://hal.science/hal-04789352}
\fi
\nolinenumbers
\ccsdesc{Theory of computation~Categorical semantics}
\keywords{Kantorovich distance,
behavioural metrics,
Kantorovich-Rubinstein duality,
functor liftings}

\EventEditors{J\"{o}rg Endrullis and Sylvain Schmitz}
\EventNoEds{2}
\EventLongTitle{33rd EACSL Annual Conference on Computer Science Logic (CSL 2025)}
\EventShortTitle{CSL 2025}
\EventAcronym{CSL}
\EventYear{2025}
\EventDate{February 10--14, 2025}
\EventLocation{Amsterdam, Netherlands}
\EventLogo{}
\SeriesVolume{326}
\ArticleNo{29}

\newtheoremrep{proposition}[theorem]{Proposition}
\newtheoremrep{lemma}[theorem]{Lemma}
\newtheoremrep{corollary}[theorem]{Corollary}

\newcommand{\Dd}{\mathcal{D}}

\newcommand{\Ll}{\mathcal{L}}
\newcommand{\Pp}{\mathcal{P}}
\newcommand{\Vv}{\mathcal{V}}
\newcommand{\cate}[1]{\text{\normalfont\textbf{#1}}}
\newcommand{\Set}{\cate{Set}}

\newcommand{\ev}{\tau}

\newcommand{\Id}{\mathrm{Id}}
\newcommand{\id}{\mathrm{id}}

\newcommand{\de}{\mathrm{d}_e}

\newcommand{\PT}{\forall}
\newcommand{\SSI}{\Leftrightarrow}

\newcommand{\inc}{\mathrm{inc}}
\newcommand{\sdw}{\mathrm{sdw}}
\newcommand{\op}{\mathrm{op}}

\newcommand{\2}{\mathsf{2}}
\newcommand{\E}{\mathbb{E}}
\newcommand{\R}{\mathbb{R}}

\renewcommand{\hom}[1]{\kl[hom]{[}#1\kl[hom]{]}}
\newcommand{\corresp}[3]{\kl[correspondence]{\left(#1,#2,#3\right)}}
\newcommand{\duality}[3]{\kl[duality]{\left(#1,#2,#3\right)}}

\ifcomments \newcommand\samuel[1]{\marginpar{sh: #1}}
\newcommand{\jr}[1]{\marginpar{Jurriaan: #1}}
\else \newcommand{\samuel}[1]{}
\newcommand{\jr}[1]{}
\fi

\begin{document}
\maketitle

\begin{abstract}
  The Kantorovich distance is a widely used metric between probability
  distributions. The Kantorovich-Rubinstein duality states that it can
  be defined in two equivalent ways: as a supremum, based on
  non-expansive functions into $[0,1]$, and as an infimum, based on
  probabilistic couplings.

  Orthogonally, there are categorical generalisations of both
  presentations proposed in the literature, in the form of
  \emph{codensity liftings} and what we refer to as
  \emph{coupling-based liftings}. Both lift endofunctors on the
  category $\Set$ of sets and functions to that of pseudometric
  spaces, and both are parameterised by modalities from coalgebraic
  modal logic.

  A generalisation of the Kantorovich-Rubinstein duality has been more
  nebulous---it is known not to work in some cases. In this paper we
  propose a compositional approach for obtaining such generalised
  dualities for a class of functors, which is closed under coproducts
  and products. Our approach is based on an explicit construction of
  modalities and also applies to and extends known cases such as that
  of the powerset functor.
\end{abstract}

\section{Introduction}

The \emph{Kantorovich} (or \emph{Wasserstein}, or \emph{Monge-Kantorovich}) distance~\cite{K42} is a standard and widely used
metric between probability distributions, studied amongst others in transportation theory~\cite{villani2009optimal}.
In concurrency theory, the Kantorovich distance forms the basis of so-called \emph{behavioural metrics},
which are quantitative generalisations of bisimilarity. They allow a more fine-grained and robust comparison
of system behaviours than classical Boolean-valued behavioural equivalences~\cite{dgjp:metrics-labelled-markov,GiacaloneEA90,bw:behavioural-pseudometric}.

In its discrete version the Kantorovich distance takes as argument a (pseudo-)metric on a set $X$,
and lifts it to a (pseudo-)metric on the set of (finitely supported) distributions $\Dd(X)$.
The celebrated \emph{Kantorovich-Rubinstein duality}~\cite{KR58} states that this distance can be computed
in two ways, yielding the same result: as an infimum indexed by probabilistic couplings, and
as a supremum indexed by non-expansive functions into the $[0,1]$ interval with the Euclidean distance.
This fundamental result is useful for analysis and computation of these distances (e.g.,~\cite{DBLP:journals/lmcs/Bacci0LM17,10.1046/j.1365-2966.2003.07106.x,DBLP:conf/fossacs/Jacobs24,DBLP:journals/siglog/Breugel17,villani2009optimal}). A detailed proof in a broader context than just finitely supported distributions can be found in~\cite{villani2009optimal}.

Orthogonally to this duality, in the last years there have been several proposals to generalise
the Kantorovich distance from distributions to general endofunctors on the category $\Set$ of sets and functions.
The problem then becomes to lift such functors to the category of (pseudo-)metric spaces.
This is particularly useful in the context of a coalgebraic presentation of systems, where the type of 
the system at hand is parametric in the given $\Set$ endofunctor. In particular, this allows a uniform presentation of various
types of probabilistic systems, but also, for instance, metrics on deterministic automata~\cite{bkp:up-to-behavioural-metrics-fibrations-journal}.

There are categorical generalisations of both presentations of the Kantorovich distance: the \emph{coupling-based} approach and 
the one based on non-expansive maps~\cite{bbkk:coalgebraic-behavioral-metrics,DBLP:conf/concur/Bonchi0P18,bkp:up-to-behavioural-metrics-fibrations-journal,HOFMANN2007789}. The latter has recently been established as an instance of the
so-called \emph{codensity liftings}\footnote{\emph{Kantorovich metric} is used interchangeably in the literature to refer to both presentations. We therefore avoid this terminology by consistently referring to the two presentations as \emph{coupling-based} and \emph{codensity} liftings respectively.}~\cite{DBLP:journals/logcom/SprungerKDH21}.
Both approaches are parametric in (sets of) \emph{modalities} or \emph{evaluation maps}, that allow a degree
of freedom in the choice of liftings.

We aim to relate the two approaches, by studying generalisations of the Kantorovich-Rubinstein duality
to a wide class of functors beyond $\Dd$. This problem was first
proposed and studied in~\cite{bbkk:coalgebraic-behavioral-metrics}, where it is shown to hold in some concrete cases
but also to fail in other basic instances, even for very elementary functors such as the diagonal functor $\Delta$ mapping a set $X$ to $X\times X$. In the latter article the authors restrict the study to liftings
parametric in exactly one modality they call an \emph{evaluation map}, which is assumed to be the same for
both codensity and coupling-based liftings. We depart from this by allowing modalities to differ on
both sides. This approach can already be found in~\cite{DBLP:conf/fossacs/GoncharovHNSW23}.
There, it is shown that every coupling-based lifting can be presented as a codensity lifting; but the proof
is non-constructive and yields a large collection of modalities.

In the current paper, we approach the problem of generalised Kantorovich-Rubinstein dualities from a concrete perspective, with an emphasis on compositionality aspects. 
Given a modality as parameter for a coupling-based lifting, we aim to explicitly translate it to modalities for
a codensity lifting, in such a way that the two correspond. More explicitly, we show that the class of such correspondences
between coupling-based and codensity liftings is closed under coproducts and products (and conversely, that every correspondence
for a coproduct of functors can be recovered from correspondences on its constituents). We also investigate correspondences for the identity functor, where there is flexibility in the choice of modalities; and for the powerset functor, extending earlier correspondence results~\cite{bbkk:coalgebraic-behavioral-metrics,DBLP:journals/acs/HofmannN20}.

These correspondence results then allow us to define a concrete \emph{grammar} of functors for which we obtain a correspondence
between coupling-based and codensity liftings. In fact, we obtain several grammars based on different assumptions on the underlying
poset of truth values (assumed to be a quantale). Base cases include the constant functors, distribution functor, powerset functor and the identity functor;
recursive constructions the product and coproduct. These results allow us to obtain or recover induced codensity and coupling-based presentations for a range of examples of behavioural metrics, including metrics on streams, labelled Markov chains, deterministic automata,
and non-deterministic automata.

In the last part of the paper, we investigate the limitations of our
approach through the concrete example of \emph{conditional transition
  systems}~\cite{DBLP:journals/lmcs/Beohar0k0W18,BEOHAR2020102320}. In contrast to the earlier examples, here, our
grammar does give us a metric (and a correspondence result), but it is
not the one considered earlier in the literature. We prove that, in fact, 
the metric from the literature can be expressed with a
codensity lifting but not with a coupling-based lifting.

\section{Codensity and coupling-based liftings: correspondences by example}
\label{sec:overview}

\AP In this section we motivate and describe the problem of correspondences between codensity and coupling-based liftings at the general level of functors,
and our approach in this paper, by means of two examples: the Kantorovich-Rubinstein duality for distributions, and
a similar correspondence for the "shortest-distinguishing-word-distance" on deterministic automata~\cite{bkp:up-to-behavioural-metrics-fibrations-journal}.

\medskip

\emph{Distributions.} We start by recalling the classical Kantorovich distance, in the discrete case. In this paper
we focus on pseudometrics (defined like metrics except that different elements can have distance $0$) as is common in the use of these
types of distances in concurrency theory. For a pseudometric $d \colon X \times X \rightarrow [0,1]$, the Kantorovich distance
is a pseudometric on the set $\Dd(X)$ of distributions on $X$, defined, for $\mu, \nu \in \Dd(X)$ by:
\begin{equation}
	\Dd^{\downarrow}(d)(\mu, \nu) = \inf_{\sigma \in \Omega(\mu, \nu)} \sum_{x,y \in X} d(x,y) \cdot \sigma(x,y)
\end{equation}
where $\Omega(\mu,\nu)$ is the set of couplings between $\mu$ and $\nu$, i.e., probability distributions on $X \times X$ whose marginals
are $\mu$ and $\nu$ respectively. It can equivalently be computed as:
\begin{equation}
	\Dd^{\uparrow}(d)(\mu, \nu) = \sup_{f \colon X \rightarrow [0,1] \text{ n.e.}} \left | \sum_{x \in X} f(x) \cdot \mu(x) - \sum_{x \in X} f(x) \cdot \nu(x) \right|
\end{equation}
where the supremum ranges over ""non-expansive functions"" $f$ into
$[0,1]$ equipped with the Euclidean distance, i.e., such that
$|f(x) - f(y)| \leq d(x,y)$ for all $x,y \in X$. The equality
$\Dd^{\uparrow}(d) = \Dd^{\downarrow}(d)$ is an instance (see~\cite[Particular case 5.16]{villani2009optimal}) of a
general result known as the Kantorovich-Rubinstein duality
(see~\cite[Theorem 5.10]{villani2009optimal}).

\medskip \emph{A different example.} The Kantorovich distance can be
seen as a \emph{lifting} of the distribution functor on the category
$\Set$ of sets and functions to the category $\cate{PMet}$ of
pseudometric spaces and non-expansive functions between them. We
proceed with a quite different example of a similar phenomenon: a
lifting of a functor from $\Set$ to $\cate{PMet}_{\le 1}$ the category
of pseudometric spaces bounded by $1$ presented in two ways, as an
infimum over a variant of couplings and a supremum over non-expansive
functions. This example describes a distance between
\emph{deterministic finite automata (DFA)}. Empty infima
and suprema are defined w.r.t.\ the interval $[0,1]$ where the pseudometrics take their values: $\sup\emptyset=0$ and
$\inf\emptyset=1$.

\footnote{The \emph{knowledge} package
  is used throughout the paper. Most
  of the introduced vocabulary and notations are clickable and the
  associated links brings the reader to their definitions.}

We view DFA over an alphabet $A$ as coalgebras for the functor
$F \colon \Set \rightarrow \Set$, $F(X) = 2 \times X^A$ where
$2 = \{0,1\}$. Coalgebras for this functor are of the form
$\langle l, \delta \rangle \colon X \rightarrow 2 \times X^A$, where
$X$ is the set of states, the output function
$l \colon X \rightarrow 2$ describes which states are accepting
($l(x)=1$), and $\delta \colon X \rightarrow X^A$ is the transition
function. As usual, a word $\omega \in A^*$ is \emph{accepted} in a state
$x$ when, after reading $\omega$ starting on $x$ we end up with an
accepting state.

\AP Given $c\in[0,1)$, the ""shortest-distinguishing-word-distance""
$"d_\sdw"(x,y)$ between states $x,y$ is $0$ if they
recognise the same language and
$c^{|\omega|}$ for $\omega$ a shortest word that belongs exactly to
one of the two languages recognised by $x$ and $y$. As shown in
\cite{bkp:up-to-behavioural-metrics-fibrations-journal}, this distance can be computed recursively as a fixpoint of
the map $\Phi$ on pseudometrics $X\times X\to [0,1]$ given by
\[
  \Phi(d)(x,y)=
  \begin{cases}
    1\text{ if }l(x)\neq l(y)\\
    c\cdot\max_{a\in A}\{d(\delta(x)(a),\delta(y)(a))\}
  \end{cases}
\]

The "shortest-distinguishing-word-distance" can be obtained via a
\emph{lifting} $\overline{F}\colon\cate{PMet}\to\cate{PMet}$ of the
$\Set$ functor $F$, mapping a pseudometric space $(X,d)$ to a
pseudometric space $(FX, \overline{F}d)$, that we will recall below.
The operator $\Phi$ above factors through the pseudometric
$\overline{F}d$. Explicitly, it decomposes as
$\Phi=\overline{F}d\circ \langle l,\delta \rangle$.

In fact, just as in the Kantorovich-Rubinstein duality,
the lifting $\overline{F}$ can be obtained in two different ways. Given a
pseudometric $d\colon X\times X\to[0,1]$ and
$(l_1,\delta_1),(l_2,\delta_2)\in 2\times X^A$, it arises as a
coupling-based lifting $\overline{F}^{\downarrow} d$ by:
\begin{equation}\label{ex:coupling-lifting-dfa}
  \overline{F}^{\downarrow}(d)((l_1,\delta_1),(l_2,\delta_2))=\inf_{(l,\delta)\in F(X\times
    X),~F\pi_i (l,\delta)=(l_i,\delta_i)} \left(\sup_{a\in
      A}c\cdot d(\delta(a))\right)
\end{equation}
Here elements of $2\times (X\times X)^A$ are viewed as a variant of couplings.

Secondly we obtain $\overline{F}$ as a so-called codensity lifting by:
\begin{equation}\label{ex:codensity-lifting-dfa}
  \overline{F}^\uparrow(d)((l_1,\delta_1),(l_2,\delta_2))=\sup_{f\colon X \to [0,1] \text{ n.e.}}
  \{|c\cdot f(\delta_1(a))-c\cdot f(\delta_2(a))|\mid a\in A\} \cup\{|l_1- l_2|\}
\end{equation}
\AP where, as above, the supremum again ranges over non-expansive functions $f$.
Just as for the Kantorovich distance, we again obtain an equality
of functors on $\cate{PMet}$: $\overline{F}^\uparrow = \overline{F}^\downarrow$.

\medskip
\emph{Modalities as parameters.} The abstract coupling-based and codensity liftings
take as input a single \emph{modality} (coupling-based), or a \emph{set of modalities} (codensity), known from the semantics of coalgebraic modal logic.
For the Kantorovich lifting, this modality is the expected value function $\mathbb{E} \colon \Dd([0,1]) \rightarrow [0,1]$, which appears implicitly in both presentations.

In the case of deterministic automata, while a single modality suffices for the
coupling-based lifting, for the codensity lifting we actually use a \emph{set} of modalities (one
for each letter, plus a modality for acceptance of states). This observation is important 
for the investigation in this paper. Instead of aiming for a one-to-one matching of modalities,
which we refer to as a \emph{duality},
we allow a coupling-based lifting specified by a single modality to be matched by a codensity lifting 
specified by a set of modalities;
we refer to this as a \emph{correspondence}.
This allows us to cover examples such as DFA and way more, and circumvent the problem already
observed in~\cite{bbkk:coalgebraic-behavioral-metrics}: even for the product functor
with the modality $\max\colon[0,1]\times[0,1]\to[0,1]$ there is no duality. However, there is a correspondence,
that is, multiple modalities are needed for the codensity lifting to match the coupling-based one.

The idea of allowing multiple modalities for
codensity liftings is not new and can be found already at the origins
of codensity liftings
\cite{DBLP:journals/lmcs/KatsumataSU18,DBLP:journals/logcom/SprungerKDH21,DBLP:conf/concur/KonigM18}.
Here we show how it can be used to generalise Kantorovich-Rubinstein
dualities in a very concrete manner; the aim is to define classes of functors and modalities
for which we can explicitly describe correspondences between associated coupling-based
and codensity liftings.

\medskip \emph{Outline.} We work at the general level of pseudometrics
valued in a quantale; we recall the definitions in
Section~\ref{sec:prelim}. We then recall the abstract notions of
coupling-based and codensity liftings of functors along modalities,
and formulate the problem of correspondences
(Section~\ref{sec:lift_dual_corr}). In
Section~\ref{sec:correspondence_abstract} we prove our main results on
correspondences: we show that the class of functors for which we have
correspondences is closed under products and coproducts, and includes
constant and identity functors, yielding a family of correspondences
for simple polynomial functors including DFA. In
Section~\ref{sec:powerset_distribution} we revisit the duality results
for the finite powerset and finite probability distribution functors.
As a consequence we are able to study in Section~\ref{sec:grammar} a
class of functors for which we have correspondences, generated by a
grammar. The case of conditional transition systems, for which we
prove that there are no correspondences possible, is treated in
Section~\ref{sec:cts}.

\section{Preliminaries: "quantales" and "pseudometrics"}
\label{sec:prelim}

We use "quantales" to model a general notion of truth object,
including both Booleans and real number intervals. In this section we
recall the necessary preliminaries on "quantales", and the associated
general notion of "pseudometrics" valued in a "quantale"; instances
include the standard notion of pseudometrics on real numbers as well
as equivalence relations.

\begin{definition}
\AP A ""quantale"" $\kl[quantale]{\Vv}$ is a
complete lattice with an associative ``and'' operation
$"\otimes"\colon"\Vv"\times"\Vv"\to"\Vv"$ which is distributive over
arbitrary joins. We only consider "quantales" that are \emph{commutative}, i.e., the
operation $"\otimes"$ is commutative, \emph{unital}, meaning $"\otimes"$ admits
a unit element, and \emph{affine}, which means that
the top element $\top$ is the unit of $"\otimes"$: $x "\otimes" \top = x = \top "\otimes" x$ for all $x \in "\Vv"$.
\end{definition}

\AP Given a "quantale" $"\Vv"$, there is an operation
$""[-,-]""\colon"\Vv"\times"\Vv"\to"\Vv"$ that is characterised by
$x"\otimes" y\le z\SSI x\le\hom{y,z}$. It is simply defined by
$\hom{y,z}=\bigvee\{x\in"\Vv"\mid x"\otimes" y\le z\}$.

\begin{example}
  \label{ex:quantales}
  \begin{itemize}
  \item \AP Any complete Boolean algebra is a "quantale" with
    $"\otimes"=\wedge$; in particular, the usual Boolean algebra
    $\2=\{\bot,\top\}$ with $\bot \leq \top$.
    In this case, $\hom{x,y}=\top$ iff $x\le y$.
  \item \AP Any interval $[0,M]$ with $M\in (0,\infty]$, given with
    \emph{reversed order} and $x"\otimes"y=\min(x+y,M)$ the \emph{truncated
      sum} is a "quantale". Here the top element is given by
    $\top = 0$, and the bottom element by $\bot = M$.
    For this "quantale", we have $\hom{x,y}=\max\{y-x,0\}$. We write
    $\overline{\R}$ for the case that $M = \infty$, i.e., the
    non-negative real numbers extended with a top element.
  \end{itemize}
\end{example}

\AP We use "quantales" as an abstract notion of truth object; accordingly, we define
a ""$\Vv$-predicate"" on a set $X$ to be a map $p\colon X\to "\Vv"$.
Of particular interest are "$\Vv$-pseudometrics". These are predicates on $X \times X$
that are reflexive, symmetric and transitive in a suitable sense that can be expressed
at the general level of quantales.

\begin{definition}
\AP A ""$\Vv$-pseudometric"" on a set $X$ is a map
$d\colon X\times X\to"\Vv"$ which is:
\begin{itemize}
\item \emph{reflexive}: $d(x,x)=\top$ for all $x\in X$,
\item \emph{symmetric}: $d(x,y)=d(y,x)$ for all $x,y\in X$,
\item \emph{transitive}:
  $\bigvee_{z\in X}d(x,z)"\otimes" d(z,y)\le d(x,y)$
  for all $x,y \in X$.
\end{itemize}
\end{definition}

\begin{example}
  \begin{itemize}
  \item $\2$-pseudometrics are equivalence relations.
  \item $\overline{\R}$-pseudometrics are the ``usual'' pseudometrics, that is, maps
  $d \colon X \times X \rightarrow \overline{\R}$ such that $d(x,x) = 0$,
  $d(x,y)=d(y,x)$ and $d(x,y) \leq d(x,z) + d(z,y)$ for all $x,y,z \in \overline{\R}$. To see
  why we obtain the triangle inequality from transitivity, recall that the order on the "quantale" is reversed w.r.t.\ the usual
  order on real numbers, and that $"\otimes" = +$.
  \end{itemize}
\end{example}

The order on a "quantale" $"\Vv"$ is extended
pointwise to functions:
\[
  \PT f,g\colon X\to"\Vv",~f\le g\SSI(\PT x\in X,~f(x)\le g(x))
\]

\AP Given two "$\Vv$-pseudometrics" $d_X$ and $d_Y$ on sets $X$
and $Y$ respectively, a map $f\colon X\to Y$ is a ""($\Vv$-pseudometric) morphism"" from $d_X$ to $d_Y$ if
$d_X\le d_Y\circ(f\times f)$.
We write $"\Vv\cate{-PMet}"$ for the category
whose objects are "$\Vv$-pseudometrics" and arrows "morphisms" between them.

\begin{remark}\label{rem:non-expansive}
 	Because of the reversal of the order on $\overline{\R}$, "morphisms" of
	$\overline{\R}$-pseudometrics are "non-expansive" maps.
	To avoid confusion with the "quantale" definition, we refrain from using the word
	\emph{"non-expansive"} altogether, and replace it by \emph{"($\Vv$-pseudometric) morphism"} instead.
\end{remark}

\AP The following canonical "$\Vv$-pseudometric" is essential for
the notion of "codensity lifting".
\begin{definition}
The ""Euclidean pseudometric"" $"\de"\colon"\Vv"\times"\Vv"\to"\Vv"$
is defined as follows:
\[
  "\de"(x,y)=\hom{x,y}\wedge\hom{y,x}
\]
\end{definition}
\begin{example}
  \begin{itemize}
  \item For the "quantale" $\2$, the "Euclidean pseudometric" $"\de" \colon \2 \times \2 \rightarrow \2$ sends equal elements
    to $\top$ and different ones to $\bot$.
  \item For the "quantale" $\overline{\R}$, the "Euclidean pseudometric" $"\de" \colon \overline{\R} \times \overline{\R} \rightarrow \overline{\R}$
  instantiates to the usual Euclidean distance, i.e.,
    $"\de"(x,y)=|y-x|$.
  \end{itemize}
\end{example}

\section{"Liftings", "dualities", and "correspondences"}
\label{sec:lift_dual_corr}

In this section we recall the definitions of "coupling-based" and "codensity liftings"
from the literature (Section~\ref{sec:liftings}) and define the problem of their "correspondence". This is followed by a few technical tools that we use in
the proofs of "correspondence" (Section~\ref{sec:kr-tools}). We start with the notion of ("well-behaved") "modality" (Section~\ref{sec:modalities}),
used in these liftings.

\subsection{"Well-behaved" "modalities"}\label{sec:modalities}

\AP Given a $\cate{Set}$ endofunctor $F$,
a ""modality"" for $F$ is a function $"\ev"\colon F"\Vv"\to"\Vv"$.
Modalities are standard in the semantics of coalgebraic modal logic~\cite{DBLP:conf/fossacs/Schroder05}.
Note that here we assume $"\Vv"$ to be a "quantale", to have a suitable notion of "pseudometrics".

\AP In order for "coupling-based liftings" to be well-defined some
particular conditions on the functor and the associated "modalities" are needed.
The underlying functor $F$ is assumed to preserve weak pullbacks.
In this context, we say $"\ev"$ is
""well-behaved""~\cite{bbkk:coalgebraic-behavioral-metrics} when:
\begin{itemize}
\item \AP it is ""monotone"", meaning that for all "$\Vv$-predicates"
  $p\le q$ we have $"\ev"\circ F(p)\le"\ev"\circ F(q)$,
\item for all "$\Vv$-predicates" $p$ and $q$,
  $"\ev"\circ F(p"\otimes" q) \ge ("\ev"\circ Fp)"\otimes"("\ev"\circ Fq)$,
\item with $i\colon\{\top\}\hookrightarrow"\Vv"$ the inclusion map,
  $Fi (F\{\top\})="\ev"^{-1}(\top)$.
\end{itemize}

\begin{example}\label{ex:well-behaved}
  \begin{itemize}
  \item On the "identity functor", with the "quantale" $\overline{\R}$,
    a "modality" is just a map $"\ev" \colon "\Vv" \to "\Vv"$. It is "well-behaved"
	if and only if it is monotone, subadditive (i.e.,
	$"\ev"(r+s) \leq "\ev"(r) + "\ev"(s)$), and satisfies $"\ev"(0) = 0$.
  \item \AP For the ""finite powerset functor"" $"\Pp"$ mapping a set
    $X$ to the set of its finite subsets $"\Pp"(X)$, with the same
    "quantale" $\overline{\R}$, the maximum function is "well-behaved".
  \item \AP For the ""finite distribution functor"" $"\Dd"$ mapping a
    set $X$ to the set of its finitely supported probability
    distributions $"\Dd"(X)$, with the "quantale" $[0,1]$, the map
    $""\E"" \colon \Dd([0,1]) \rightarrow [0,1]$ giving the expected value of a probability distribution is
    "well-behaved".
  \end{itemize}
\end{example}

In the rest of this paper we consider "constant", "identity",
"finite powerset", "finite probability distribution functors", and
combine them using products and coproducts. All
these functors preserve weak pullbacks (see~\cite[Propositions 4.2.6 and
4.2.10]{DBLP:books/cu/J2016} for example), ensuring we can consider
"well-behaved" "modalities" for them.

"Well-behaved" "modalities" can be constructed from
existing ones. The case of composition is a generalisation to
"quantales" of a particular case of~\cite[Theorem
7.2]{bbkk:coalgebraic-behavioral-metrics}.
\begin{propositionrep}\label{prop:wb_evmap_stability}
  Given three arbitrary "well-behaved" "modalities"
  $"\ev","\ev"'\colon F"\Vv"\to"\Vv"$ and
  $"\ev"_{"\Id"}\colon "\Vv"\to"\Vv"$, 
  the following "modalities" are again "well-behaved":
  $
    "\ev"_{"\Id"}\circ"\ev"$, $"\ev" "\otimes" "\ev"'$ and $"\ev"\wedge"\ev"'$. 
\end{propositionrep}
\begin{proof}
  Let $p,q$ be "$\Vv$-predicates" such that $p\le q$:
  \begin{itemize}
  \item "Monotonicity": by hypothesis
    \[
      "\ev"_{"\Id"}\circ"\Id"("\ev"\circ Fp)\le"\ev"_{"\Id"}\circ"\Id"("\ev"\circ
      Fq)
    \]
    directly giving "monotonicity" for $"\ev"_{"\Id"}\circ"\ev"$.

    "Monotonicity" of $"\ev"~\op~"\ev"'$ for
    $\op\in\{"\otimes",\bigwedge\}$ is given by that of $"\ev"$,
    $"\ev"'$, and $\op$ (in both its arguments).
  \item Second condition: as $"\ev"$ and $"\ev"_{"\Id"}$ are both "well-behaved"
    \begin{align*}
      ("\ev"_{"\Id"}\circ"\ev")\circ F(p"\otimes" q)
      &="\ev"_{"\Id"}("\ev"\circ F(p"\otimes" q))\\
      &\ge"\ev"_{"\Id"}(("\ev"\circ Fp)"\otimes"("\ev"\circ Fq))\\
      &="\ev"_{"\Id"}\circ"\Id"(("\ev"\circ Fp)"\otimes"("\ev"\circ Fq))\\
      &\ge("\ev"_{"\Id"}\circ"\ev")Fp"\otimes"("\ev"_{"\Id"}\circ"\ev")Fq
    \end{align*}

    By hypothesis both $"\ev"$ and $"\ev"'$ are "monotone". As
    $"\otimes"$ is monotone in both its arguments:
    \[
      ("\ev"\circ F(p"\otimes" q))"\otimes"("\ev"'\circ F(p"\otimes"
      q))\ge(("\ev"\circ Fp)"\otimes"("\ev"\circ
      Fq))"\otimes"(("\ev"'\circ Fp)"\otimes"("\ev"'\circ Fq))
    \]
    and the associativity and commutativity of $"\otimes"$ allow us to
    conclude.

    Finally, for $"\ev"\bigwedge"\ev"'$:
    \begin{align*}
      ("\ev"\bigwedge"\ev"')\circ F(p"\otimes" q)
      &=("\ev"\circ F(p"\otimes" q))\bigwedge("\ev"'\circ F(p"\otimes" q))\\
      &\ge(("\ev"\circ Fp)"\otimes"("\ev"\circ Fq))\bigwedge
        (("\ev"'\circ Fp)"\otimes"("\ev"'\circ Fq))\\
      &\ge(("\ev"\circ Fp)\bigwedge"\ev"'\circ Fp)"\otimes"
        (("\ev"\circ Fq)\bigwedge"\ev"'\circ Fq)
    \end{align*}
    where the last inequality is given by monoticity of $"\otimes"$.
  \item Last condition:
    \begin{align*}
      ("\ev"_{"\Id"}\circ"\ev")^{-1}(\top)
      &= "\ev"^{-1}("\ev"_{"\Id"}^{-1}(\top))\\
      &= "\ev"^{-1}(\top)\text{ (as $"\ev"_{"\Id"}$ is "well-behaved")}\\
      &= Fi(F\{\top\})\text{ (as $"\ev"$ is "well-behaved")}
    \end{align*}

    For $"\ev" "\otimes" "\ev"'$, note that for all $x,y$,
    $x=x"\otimes"\top\ge x"\otimes" y$, so that $x"\otimes" y=\top$ if and
    only if $x=\top=y$. The same holds replacing $"\otimes"$ by $\bigwedge$
    so that what we do here proves the condition for
    $"\ev"\bigwedge"\ev"'$ too. We get
    $("\ev" "\otimes" "\ev"')^{-1}(\top)=
    ("\ev"^{-1}(\top))\cap("\ev"'^{-1}(\top))=Fi(F\{\top\})$.
  \end{itemize}
\end{proof}

\subsection{"Liftings" and "correspondences"}\label{sec:liftings}

\AP Given a functor $P\colon\cate{C}\to\cate{D}$, a ""lifting"" of a
$\cate{D}$-functor $F\colon\cate{D}\to\cate{D}$ from $\cate{D}$ to
$\cate{C}$ is a functor $\overline{F}\colon\cate{C}\to\cate{C}$ such
that $P\circ\overline{F}=F\circ P$. Here we only consider the case
when $\cate{D}=\Set$ and $\cate{C}="\Vv\cate{-PMet}"$, with $P$ the
forgetful functor sending a "$\Vv$-pseudometric" of type
$X\times X\to"\Vv"$ to its underlying set $X$ and acting as the identity on
arrows. We define "coupling-based liftings", "codensity liftings", and
the associated notion of "correspondence" that we study here.

\AP  Given
$t_1,t_2\in FX$, a ""coupling"" of $t_1$ and $t_2$ is an element
$t\in F(X\times X)$ such that $F\pi_i(t)=t_i$. The set of "couplings"
of $t_1$ and $t_2$ is denoted by $"\Omega"(t_1,t_2)$.
In particular, if $F = "\Dd"$ is the "distribution functor" on $\Set$, these are precisely
probabilistic couplings: joint distributions on $X \times X$ whose marginals
coincide with given distributions $t_1, t_2$. The following "lifting"
arises from~\cite{HOFMANN2007789}. See also~\cite{bkp:up-to-behavioural-metrics-fibrations-journal}. It is not referred to in the literature as ``"coupling-based"''; we use this terminology
to distinguish it from the "codensity lifting".

\begin{definition}\label{def:coupling-based-lifting}
  \AP Let $F \colon \Set \rightarrow \Set$ be a functor which preserves weak pullbacks, and let
  $"\ev"$ be a "well-behaved" "modality" for $F$.
  The ""coupling-based
  lifting"" of $F$ to $"\Vv\cate{-PMet}"$ is given by, with
  $d\in"\Vv\cate{-PMet}"$ and $t_1,t_2\in FX$:
  \[
    \kl[coupling-based]{F^\downarrow_{"\tau"}}(d)(t_1,t_2)=\bigvee_{t\in "\Omega"(t_1,t_2)}"\ev"\circ Fd(t)
  \]
\end{definition}

\begin{example}
In Section~\ref{sec:overview}, we studied a coupling-based "lifting" for DFA, which yields the
"shortest-distinguishing-word-distance". To see this as an instance of Definition~\ref{def:coupling-based-lifting},
the "quantale" is $"\Vv"=[0,1]$ as introduced in Section~\ref{sec:prelim}, the functor $F$ is the one mapping a set $X$ to
$2\times X^A$. The "well-behaved" "modality" takes as argument an
element $(l,\delta)\in 2\times "\Vv"^A$ and returns
$\bigwedge_{a\in A}c\cdot\delta(a)$ for some $c\in[0,1)$. The resulting "lifting"
is precisely the one given in Equation~\eqref{ex:coupling-lifting-dfa}.
\end{example}

The \emph{"codensity lifting"} is defined in a very general setting for monads in~\cite{DBLP:journals/lmcs/KatsumataSU18}. Here, we use an instance,
which appears in~\cite{DBLP:journals/logcom/SprungerKDH21,DBLP:conf/concur/KonigM18} and, for a single "modality", in~\cite{bbkk:coalgebraic-behavioral-metrics}.

\begin{definition}
  \AP Let $F \colon \Set \rightarrow \Set$, and
  let $\Gamma$ be a family of "modalities" for $F$.
  The ""codensity lifting"" of $F$ along $\Gamma$ to
  $"\Vv\cate{-PMet}"$ is defined by, with $d\in"\Vv\cate{-PMet}"$
  and $t_1,t_2\in FX$:
  \[
    \kl[codensity]{F^\uparrow_\Gamma}(d)(t_1,t_2)=\bigwedge_{\tau\in\Gamma,~f\colon d\to"\de"}
    "\de"("\ev"\circ Ff(t_1),"\ev"\circ Ff(t_2))
  \]
\end{definition}
Note that the maps $f\colon d\to"\de"$ in the definition of the
"codensity lifting" are "$\Vv$-pseudometric morphisms", so that, for
instance, in $\overline{\R}$, they correspond to "non-expansive" maps
(cf.~Remark~\ref{rem:non-expansive}).

\begin{example}
  In the example of Section~\ref{sec:overview} the "codensity lifting"
  has one "modality" mapping $(l,\delta)\in 2\times"\Vv"^A$ to $l$, as
  well as one "modality" $"\tau_a"$ for each letter $a\in A$ given by
  $"\tau_a"(l,\delta)=c\cdot\delta(a)$. The resulting "codensity
  lifting" coincides with the "lifting" given in
  Equation~\eqref{ex:codensity-lifting-dfa}.
\end{example}

Note that the "coupling-based liftings" are only allowed
to have one "modality", whereas the "codensity" ones may have multiple
"modalities". Informally, one of the reasons "coupling-based liftings"
do not require multiple "modalities" can be seen by looking again at
the functor for DFA. Recall that it maps a set $X$ to $\2\times X^A$
for $\2=\{0,1\}$ and $A$ an alphabet. Considering "couplings" forces
comparisons using $d$ to be done separately on each letter. "Codensity
liftings" instead use one "modality" per letter to ensure a similar
condition. On a higher level of abstraction, the multiple "modalities"
of "codensity liftings" relate to the expressivity of some modal logics
(see, e.g.,~\cite{kkkrh:expressivity-quantitative-modal-logics}). On the
other hand, the unique "modality" for "coupling-based liftings" has
often been called an evaluation function (see~\cite{bbkk:coalgebraic-behavioral-metrics}). Having only one way of
evaluating something seems natural in a given evaluation context.

In the rest of this paper we study possible equalities between the
"coupling-based" and "codensity liftings". In general, we allow
a "well-behaved" "modality" for the "coupling-based liftings" to be matched
by a set of (different) "modalities" for the "codensity lifting". In case
the "modalities" are the same, we call this "Kantorovich-Rubinstein duality".
\begin{definition}\label{def:correspondence}
  \AP Let $F\colon\Set\to\Set$ be a weak-pullback preserving functor,
  $"\ev"\colon F"\Vv"\to"\Vv"$ a "well-behaved" "modality", and
  $\Gamma$ a set of "modalities" for $F$. A ""correspondence"" is a
  triple of the form $\corresp{F}{"\ev"}{\Gamma}$ such that the
  associated "coupling-based" and "codensity" "liftings" coincide,
  as in:
  \[
    \kl[codensity]{F^\downarrow_{"\ev"}}=\kl[codensity]{F^\uparrow_\Gamma}
  \]
  If $\Gamma = \{"\ev"\}$ then we refer to this as ""Kantorovich-Rubinstein duality"", or simply "duality".
\end{definition}
\AP Two "correspondences" $\corresp{F}{"\ev_1"}{\Gamma_1}$ and
$\corresp{F}{"\ev_2"}{\Gamma_2}$ for the same functor are called
""equivalent"" when they yield the same "liftings":
\[
  \kl[coupling-based]{F^\downarrow_{"\ev_1"}}=\kl[codensity]{F^\uparrow_{\Gamma_1}}=
  \kl[codensity]{F^\uparrow_{\Gamma_2}}=\kl[coupling-based]{F^\downarrow_{"\ev_2"}}
\]
This is denoted by
$\corresp{F}{"\ev_1"}{\Gamma_1}"\sim"\corresp{F}{"\ev_2"}{\Gamma_2}$.

When the associated "liftings" are not equal but one ""inequality""
holds nonetheless such as
\[
  \kl[coupling-based]{F^\downarrow_{"\ev_1"}}=\kl[codensity]{F^\uparrow_{\Gamma_1}}\le
  \kl[codensity]{F^\uparrow_{\Gamma_2}}=\kl[coupling-based]{F^\downarrow_{"\ev_2"}}
\]
we write
$\corresp{F}{"\ev_1"}{\Gamma_1}"\le"\corresp{F}{"\ev_2"}{\Gamma_2}$.

\subsection{"Kantorovich-Rubinstein dualities" and tools to get them}\label{sec:kr-tools}

\AP The focus of this paper is on general "correspondences" between "coupling-based" and "codensity liftings",
but in some cases there is the stronger property of "Kantorovich-Rubinstein
duality", meaning "correspondences" like
$\duality{F}{"\ev"}{\{"\ev"\}}$. In the remainder of this section we consider a few technical
definitions and results that are useful in obtaining such "duality" results.

First, "coupling-based liftings" are always smaller than
"codensity liftings". This is a known result; it is a direct consequence of
"codensity liftings" being initial in a well-chosen category as shown
in~\cite{DBLP:conf/fossacs/GoncharovHNSW23}. It can also be found in
\cite[Theorem 5.27]{bbkk:coalgebraic-behavioral-metrics} in the
particular case of the "quantale" $[0,M]$ as presented in Section~\ref{sec:prelim}.
\begin{propositionrep}\label{prop:weak_duality}
  Given a $\cate{Set}$ endofunctor $F$ and an associated
  "well-behaved" "modality", we have
  $
    \kl[coupling-based]{F^\downarrow_{"\ev"}}\le \kl[codensity]{F^\uparrow_{\{"\ev"\}}}
  $.
\end{propositionrep}
\begin{proof}
  We want to prove that given any "$\Vv$-pseudometric" $d\colon X\times X\to"\Vv"$
  and $t_1,t_2\in FX$ we have
  \[
    \bigvee_{t\in "\Omega"(t_1,t_2)}
    "\ev"\circ Fd(t)
    \le
    \bigwedge_{f\colon d\to"\de"}
    "\de"("\ev"\circ Ff(t_1),"\ev"\circ Ff(t_2))
  \]
  To do that we only need proving that for every $f\colon d\to"\de"$ and
  $t\in"\Omega"(t_1,t_2)$ we have
  \[
    "\ev"\circ Fd(t)\le"\de"("\ev"\circ Ff(t_1),"\ev"\circ Ff(t_2))
  \]
  We have the following:
  \begin{align*}
    "\de"("\ev"\circ Ff(t_1),"\ev"\circ Ff(t_2))
    &="\de"("\ev"\circ Ff(F\pi_1(t)),"\ev"\circ Ff(F\pi_2(t_2)))\\
    &="\de"("\ev"\circ F(f\circ\pi_1)(t),"\ev"\circ F(f\circ\pi_2)(t_2))\\
    &="\de"("\ev"\circ F(\pi_1\circ(f\times f))(t),"\ev"\circ F(\pi_2\circ(f\times f))(t_2))\\
    &="\de"("\ev"\circ F\pi_1(F(f\times f)(t)),"\ev"\circ F\pi_2(F(f\times f)(t_2)))\\
  \end{align*}

  For $i,j\in\{1,2\}$ and $i\neq j$, by definition of $"\de"$ we have
  $"\de"\le\hom{\pi_i,\pi_j}$ and hence $"\de" "\otimes"\pi_i\le\pi_j$.
  Applying $F$ and then $"\ev"$,
  $"\ev"\circ F("\de" "\otimes"\pi_i)\le"\ev"\circ F(\pi_j)$. Because $"\ev"$ is
  "well-behaved" this gives
  $"\ev"\circ F("\de" "\otimes"\pi_i)\ge"\ev"\circ F("\de")"\otimes" "\ev"\circ
  F(\pi_i)$ so that
  $"\ev"\circ F("\de")"\otimes" "\ev"\circ F(\pi_i)\le"\ev"\circ F(\pi_j)$ and
  then $\hom{"\ev"\circ F(\pi_i),"\ev"\circ F(\pi_j)}\ge "\ev"\circ F("\de")$ and
  finally
  $"\de"(("\ev"\circ F\pi_1)(t),("\ev"\circ F\pi_2)(t))\ge("\ev"\circ
  F)"\de"(t)$.
\end{proof}

Hence to get a "Kantorovich-Rubinstein duality", we only need the
other inequality. Two tools are introduced for that:
"optimal couplings" to know the value of "coupling-based liftings" and
"optimal functions" to know that of "codensity liftings".

\AP Given a functor $F \colon \Set \rightarrow \Set$, an associated
"modality" $"\ev"$, a "$\Vv$-pseudometric"
$d\colon X\times X\to"\Vv"$, and $t_1,t_2\in FX$, an ""optimal coupling""
is a "coupling" $t\in F(X\times X)$ of $t_1$ and $t_2$ such that
\[
  \kl[coupling-based]{{F}^\downarrow_{"\ev"}}d(t_1,t_2)="\ev"\circ Fd(t)
\]
When for all $d$, $t_1$, and $t_2$ either an "optimal coupling" exists
or no "couplings" of $t_1$ and $t_2$ exist we say that $F$ \emph{has all
optimal couplings}. "Optimal couplings" are well-known in the context
of transport theory~\cite{villani2009optimal} and in
existing proofs of "duality" for the "finite powerset functor"
\cite{DBLP:journals/acs/HofmannN20,bbkk:coalgebraic-behavioral-metrics}.
We note that, whenever all elements in a class of functors $\mathcal{F}$ have all "optimal
couplings", then functors in the closure of $\mathcal{F}$ by coproducts
and products do too.

\AP An ""optimal function"" for a functor $F$, an associated
"modality" $"\ev"$, a "$\Vv$-pseudometric"
$d\colon X\times X\to"\Vv"$, and $t_1,t_2\in FX$ is a morphism in
$"\Vv\cate{-PMet}"$ $f\colon d\to"\de"$ such that:
\[
  \kl[codensity]{F^\uparrow_{\{"\ev"\}}}d(t_1,t_2)="\de"("\ev"\circ Ff(t_1),"\ev"\circ Ff(t_2))
\]

In the context of the general continuous Kantorovich-Rubinstein duality from optimal transport theory such "optimal functions"
do exist (see~\cite[proof of Theorem 5.10]{villani2009optimal}) but
are not explicitly defined. As we will see they exist for all the functors
we consider.

The following lemma is useful to design "optimal functions". It is
well-known in the context of "quantale"-enriched categories. Instead
of defining some "pseudometric morphism" $f\colon d\to"\de"$ on $X$ as a
whole, it can be defined on $Y\subseteq X$ only, first as some
"morphism" $g\colon d\circ i\to"\de"$ and then extended using the
lemma. This is useful when only parts of $X$ are of interest, say some
subset of it for the "powerset functor", or the support of some
probability distribution for the "distribution functor". The extension
of $g$ to $f$ from $Y$ to $X$ is done by giving the greatest values to $f$ on $X\backslash Y$ while
ensuring it is a "$\Vv$-pseudometric morphism".
\begin{lemmarep}
  \label{lem:extension_of_V-Rel_morphisms}
  Let $d\colon X\times X\to"\Vv"$ be a "$\Vv$-pseudometric", and
  $Y\overset{i}{\subseteq} X$. For all "$\Vv$-pseudometric" morphisms
  $g\colon d\circ (i\times i)\to"\de"$ there exists $f\colon d\to"\de"$ s.t.\
  $
    g=f\circ i
  $
  and $f$ is the least such morphism.
\end{lemmarep}
\begin{proof}
	We define
  $f\colon d\to"\de"$ by the following:
  \[
    \forall x\in X,~f(x)=\bigvee\left\{ g(u)"\otimes" d(x,u)\mid u\in Y \right\}
  \]
  The remainder of the proof is routine.
\end{proof}

\section{"Correspondences" through coproducts and products}
\label{sec:correspondence_abstract}

In this section we obtain new "correspondences" between "coupling-based" and "codensity liftings",
for polynomial functors. More specifically, given a set of functors $\mathcal{F}$ and
associated "correspondences"
$\corresp{F}{"\ev_F"}{\Gamma_F}$, these are extended to "correspondences" for the coproduct (Section~\ref{sec:coproduct})
and product (Section~\ref{sec:product}) of the underlying functors. In Section~\ref{sec:correspondence_basic} we study "correspondences"
for "constant" and "identity functors". Finally in Section~\ref{sec:polynomial} we combine these results to
describe several collections of functors for which we have "correspondences". Throughout this section
we fix a "quantale" $"\Vv"$.

\subsection{Coproduct functors}
\label{sec:coproduct}

\AP We start by showing that "correspondences" are closed under coproducts.
Given sets $S_i$ we write $"\kappa_j"$ for the $j$th
""coprojection"" $"\kappa_j"\colon S_j\to\coprod S_i$.
\begin{propositionrep}\label{prop:duality_coprod}
  Given "correspondences" $\corresp{F_i}{"\ev_i" \colon F_i "\Vv" \to "\Vv"}{\Gamma_i}$ there is a
  "correspondence"
  \[
    \corresp{\coprod F_i}{["\ev_i"]}{\bigcup
      \overline{\Gamma_i}}
	  \qquad \text{ where } 
      \overline{\Gamma_i}=
      \left\{
        \overline{"\ev"} \mid "\ev"\in \Gamma_i
      \right\}
      \cup\{"\ev"_{\top,i}\};
  \]
  the map $["\ev_i"] \colon \coprod F_i "\Vv" \rightarrow "\Vv"$ is the cotupling of the individual modalities; for
  $"\ev"\in \Gamma_i$,
  \[
    \overline{"\ev"}(x)=
    \begin{cases}
      "\ev"(y)\text{ when there exists $y\in F_i"\Vv"$ and $x="\kappa_i" y$}\\
      \top\text{ otherwise}
    \end{cases}
  \]
  and finally,
  $
    "\ev"_{\top,i}(x)=
    \begin{cases}
      \top\text{ when there exists $y\in F_i"\Vv"$ and $x="\kappa_i" y$}\\
      \bot\text{ otherwise}
    \end{cases}
  $.
\end{propositionrep}
\begin{proof}
  Throughout the proof we let $F=\coprod F_i$.
  Assume
  "correspondences" $\corresp{F_i}{"\ev_i"}{\Gamma_i}$. Let
  $d\colon X\times X\to"\Vv"$ be a "$\Vv$-pseudometric" and
  $x,y\in\left( \coprod F_i \right)X$ such that $x="\kappa_j" x'$ and
  $y="\kappa_k" y'$.

  We prove the "correspondence" by case analysis on whether $j=k$ or
  not.

  \medskip

  If $j\neq k$ then there are no "couplings" of $x$ with $y$, giving
  directly:
  \[
    \kl[coupling-based]{{\left( \coprod F_i
        \right)}^{\downarrow}_{\left[ "\ev_i" \right]}}d(x,y)=\bot
  \]
  Consider the "codensity lifting" of
  $F=\left( \coprod F_i \right)$ along $"\ev"_{\top,j}$ only.
  For every $f\colon d\to"\de"$ we get $Ff x="\kappa_j" (F_jf x')$ and
  $Ff y="\kappa_k" (F_kf x')$. Thus, $"\ev"_{\top,j}(Ff(x))=\top$ while
  $"\ev"_{\top,j}(Ff(y))=\bot$. The "quantale" being unital
  $\bot"\otimes"\top=\bot$ which implies $\hom{\top,\bot}=\bot$ and
  $"\de"(\bot,\top)=\bot$. This gives
  $"\de"("\ev"_{\top,j}\circ Ff(x),"\ev"_{\top,j}\circ Ff(y))=\bot$. As
  "codensity liftings" are defined by meets,
  $\kl[codensity]{{F}^{\uparrow}_{\bigcup\overline{\Gamma_i}}}d(x,y)=\bot$ and
  $\kl[codensity]{{F}^{\uparrow}_{\bigcup\overline{\Gamma_i}}}d(x,y)=\kl[coupling-based]{{F}^{\downarrow}_{\left[
      "\ev_i" \right]}}d(x,y)$.

  \medskip

  Otherwise assume that $j=k$. "Couplings" $z$ of $x$ with $y$ are
  exactly given by $"\kappa_j" z'$ for $z'$ "couplings" of $x'$ and $y'$.
  \begin{align*}
    \kl[coupling-based]{{F}^{\downarrow}_{\left[ "\ev_i" \right]}}d(x,y)
    &= \bigvee_{z\in"\Omega"(x,y)} \left[ "\ev_i" \right]\circ Fd(z)\\
    &= \bigvee_{z'\in"\Omega"(x',y')} \left[ "\ev_i" \right]\circ Fd("\kappa_j" z')\\
    &= \bigvee_{z'\in"\Omega"(x',y')} "\ev_j"\circ F_jd(z')\\
    &= \kl[coupling-based]{{F_j}^{\downarrow}_{"\ev_j"}}d(x',y')\\
    &= \kl[codensity]{{F_j}^{\uparrow}_{\Gamma_j}}d(x',y')\text{ (by hypothesis)}
  \end{align*}

  We now look at the value of the "codensity lifting".
  Consider $"\ev"\in\bigcup\overline{\Gamma_i}$.

  Suppose $"\ev"="\ev"_{\top,n}$ for some $n$. Whether $n=j$ or $n\neq j$,
  for all $f\colon d\to"\de"$, $"\ev"\circ Ff(x)="\ev"\circ Ff(y)$ are
  either both equal to $\top$ or both equal to $\bot$ and
  \[
    \kl[codensity]{{F}^{\uparrow}_{\{"\ev"_{\top,n}\}}}d(x,y)=\top
  \]

  Now suppose $"\ev"=\overline{"\ev"'}$ for some $"\ev"'\in \Gamma_n$.
  When $n\neq j$ we again obtain $"\ev"\circ Ff(x)="\ev"\circ Ff(y)=\top$
  and
  \[
    \kl[codensity]{{F}^{\uparrow}_{\{"\ev"\}}}d(x,y)=\top
  \]

  Otherwise $n=j$. Given $f\colon d\to"\de"$,
  $"\ev"\circ Ff(x)="\ev"'\circ F_jf(x')$ and
  $"\ev"\circ Ff(y)="\ev"'\circ F_jf(y')$ so that taking the meet of
  $"\de"(\ev\circ Ff(x),\ev\circ Ff(y))$ over all morphisms $f$ gives
  \[
    \kl[codensity]{F^{\uparrow}_{\{"\ev"\}}}d(x,y)=\kl[codensity]{{F_j}_{\{"\ev"'\}}^\uparrow}
    d(x',y')
  \]

  Taking the meet over all such "modalities" gives
  \[
    \kl[codensity]{F^\uparrow_{\bigcup\overline{\Gamma_i}}}d(x,y)=\kl[coupling-based]{{F_j}^\uparrow_{\Gamma_j}}d(x',y')
  \]
\end{proof}
\smallskip
When there are no "couplings" for two given elements, as it happens for instance for elements in different components of
a coproduct, the "coupling-based lifting" always gives the value $\bot$.
The "modalities" $"\ev"_{\top,i}$ are there to ensure that the "codensity lifting"
also gives the value $\bot$ in this case. The other "modalities" are just the extensions
of the "modalities" from the components of the coproduct to the coproduct itself.

\begin{remark}
  Proposition~\ref{prop:duality_coprod} gives a "correspondence" but not a "duality", in the sense of Definition~\ref{def:correspondence},
that is, the "correspondence" relates a single "modality" for the "coupling-based lifting" to a set of "modalities" for
the "codensity lifting". This is the case even if the "correspondences" of the component functors $F_i$ are "dualities".
\end{remark}

As it turns out, the inverse process of going from a "correspondence" at the level of the coproduct to "correspondences" for its components
is also possible:
\begin{propositionrep}\label{prop:duality_from_coprod}
  If there is a "correspondence" $\corresp{\coprod F_i}{"\ev"}{\Gamma}$, then
  there are "correspondences" 
  $
    \corresp{F_i}{"\ev"\circ"\kappa_i"}{\Gamma\circ"\kappa_i"}
  $ for each $i$.
\end{propositionrep}
\begin{proof}
  Assume a "correspondence" $\corresp{\coprod F_i}{"\ev"}{\Gamma}$. Let
  $d\colon X\times X\to"\Vv"$ be a "$\Vv$-pseudometric" and
  $x,y\in F_j X$ for some $j$:
  \begin{align*}
    \kl[coupling-based]{{F_j}^\downarrow_{"\ev"\circ"\kappa_j"}}
    d(x,y)
    &= \bigvee_{z\in"\Omega"(x,y)}("\ev"\circ"\kappa_j")\circ F_jd(z)\\
    &= \bigvee_{z\in"\Omega"(x,y)}"\ev"\circ Fd("\kappa_j" z)\\
    &= \bigvee_{z'\in"\Omega"("\kappa_j" x,"\kappa_j" y)}"\ev"\circ Fd(z')\\
    &= \kl[coupling-based]{F^\downarrow_{"\ev"}} d("\kappa_j" x,"\kappa_j" y)\\
    &= \kl[codensity]{F^\uparrow_{\Gamma}}d("\kappa_j"x,"\kappa_j"y)\text{~(by hypothesis)}\\
    &= \bigwedge_{"\ev"'\in \Gamma,~f\colon d\to"\de"}"\de"("\ev"'\circ Ff("\kappa_j" x),"\ev"'\circ Ff("\kappa_j" y))\\
    &= \bigwedge_{"\ev"'\in \Gamma,~f\colon d\to"\de"}"\de"(("\ev"'\circ"\kappa_j")\circ F_jf(x),("\ev"'\circ"\kappa_j")\circ F_jf(y))\\
    &= \bigwedge_{"\ev"'\in \Gamma\circ"\kappa_j",~f\colon d\to"\de"}"\de"("\ev"'\circ
      F_jf(x),"\ev"'\circ F_jf(y))
  \end{align*}
  proving $\corresp{F_j}{"\ev"\circ"\kappa_j"}{\Gamma\circ"\kappa_j"}$ is a
  "correspondence".
\end{proof}
Going back and forth between the components of a coproduct and the
coproduct itself using the results above gives back the same
"correspondences":
\begin{propositionrep}
  \label{prop:duality_coprod_equiv}
  Whenever one of the "correspondences" on the left exists we also have the corresponding "equivalence":
  \begin{align*}
    \corresp{F_i}{"\ev_i"}{\Gamma_i}\quad&"\sim"\quad\corresp{F_i}{["\ev_j"]\circ"\kappa_i"}{\bigcup\overline{\Gamma_j}\circ"\kappa_i"}\, \text{ and}\\
    \corresp{\coprod
    F_i}{"\ev"}{\Gamma}\quad&"\sim"\quad\corresp{\coprod
    F_i}{["\ev"\circ"\kappa_i"]}{\bigcup\overline{\Gamma\circ"\kappa_i"}}\,
  \end{align*}
\end{propositionrep}
\begin{proof}
  Note that for all $i$, $["\ev_j"]\circ"\kappa_i"="\ev_i"$ and
  $["\ev"\circ"\kappa_i"]="\ev"$ so that for each possible
  "equivalence" the "coupling-based liftings" on both side coincide,
  meaning overall that all "liftings" coincide.
\end{proof}
This shows that all "correspondences" on the coproduct arises as in Proposition~\ref{prop:duality_coprod}.

\subsection{Product functors}
\label{sec:product}

We turn to the construction of "correspondences" for products from ones on their components.
For this, we ask for some kind of distributivity condition on the "quantale". In fact,
the case of products is more involved than that of coproducts. Each of
the three propositions for coproducts above has a version for products, but each
gets their own specific restrictions in the form of conditions on the "quantale" and "modalities" or even
on the statement itself.

\AP There
are several possibilities depending on the number of "couplings" and on
whether we want finite or arbitrary products. Below, we say that a functor $F$ has ""finite couplings""
if, for any $t_1, t_2 \in FX$, the set $"\Omega"(t_1,t_2) \subseteq F(X \times X)$ of "couplings" is finite.
\begin{propositionrep}\label{prop:duality_prod}
  Let $\corresp{F}{"\ev"_F}{\Gamma_F}$ and $\corresp{G}{"\ev"_G}{\Gamma_G}$ be "correspondences". If one of
  the following conditions holds:
  \smallskip
  \begin{enumerate}
  \item\label{hyp:1} $F$ and $G$ have "finite couplings" and $"\Vv"$ is distributive, meaning
    for all $x,y,z\in"\Vv"$, we have that $(x\vee z)\wedge(y\vee z)=(x\wedge y)\vee z$; or,
  \item\label{hyp:2} $"\Vv"$ is join-infinite distributive: for all $x\in"\Vv"$
    and $V\subseteq"\Vv"$, $x\wedge\bigvee V=\bigvee\{x\wedge v\mid v\in V\}$;
  \end{enumerate}
  \smallskip
  then we have a "correspondence"
  $\corresp{F\times
    G}{("\ev"_F\circ\pi_1)\wedge("\ev"_G\circ\pi_2)}{(\Gamma_F\circ\pi_1)\cup(\Gamma_G\circ\pi_2)}$.
\end{propositionrep}
\begin{proof}
  Let $X$ be a set, $d\colon X\times X\to"\Vv"$ a
  "$\Vv$-pseudometric", and $x,y\in(F\times G)X$. We note
  $x=(x_F,x_G)$ and $y=(y_F,y_G)$.

  Note that "couplings" of $x$ with $y$ are of the form $z=(z_F,z_G)$
  where $z_F$ (resp. $z_G$) is a "coupling" of $x_F$ with $y_F$ (resp.
  $x_G$ with $y_G$).

  We have:
  \begin{align*}
    \kl[coupling-based]{{\left( F\times G \right)}^{\downarrow}_{"\ev"_F\circ\pi_1\wedge"\ev"_G\circ\pi_2}}d(x,y)
    &= \bigvee_{z\in"\Omega"(x,y)} ("\ev"_F\circ\pi_1\wedge"\ev"_G\circ\pi_2)\circ (F\times G)d(z)\\
    &= \bigvee_{z_B\in"\Omega"(x_B,y_B)\text{ for }B\in\{F,G\}} ("\ev"_F\circ Fd(z_F))\wedge("\ev"_G\circ Gd(z_G))\\
    &= \bigvee_{(f,g)\in C_F\times C_G} f\wedge g
  \end{align*}
  where $C_F=\left\{ "\ev"_F\circ Fd(z_F)\mid z_F\text{ is a "coupling" of $x_F$ with $y_F$} \right\}$
  and similarly for $C_G$.

  Whether it be under
  Hypothesis~\hyperref[hyp:1]{\color{lipicsGray}{\sffamily\bfseries\upshape\mathversion{bold}1.}}
  or~\hyperref[hyp:2]{\color{lipicsGray}{\sffamily\bfseries\upshape\mathversion{bold}2.}},
  we get:
  \begin{align*}
    \bigvee_{(f,g)\in C_F\times C_G} f\wedge g
    &= \left( \bigvee C_F \right)\bigwedge\left( \bigvee C_G \right)\\
    &= \left( \kl[coupling-based]{{F}^{\downarrow}_{"\ev"_F}}d(x_F,y_F) \right)
      \bigwedge
      \left( \kl[coupling-based]{{G}^{\downarrow}_{"\ev"_G}}d(x_G,y_G) \right)\\
    &= \left( \kl[codensity]{{F}^\uparrow_{\Gamma_F}}d(x_F,y_F) \right)
      \bigwedge
      \left( \kl[codensity]{{G}^\uparrow_{\Gamma_G}}d(x_G,y_G) \right)
      \text{ (by hypothesis)}\\
    &= \left( \kl[codensity]{{(F\times G)}^\uparrow_{\Gamma_F\circ\pi_1}}d(x,y) \right)
      \bigwedge
      \left( \kl[codensity]{{(F\times G)}^\uparrow_{\Gamma_G\circ\pi_2}}d(x,y) \right)\\
    &= \kl[codensity]{{(F\times G)}^\uparrow_{\Gamma_F\circ\pi_1\bigcup \Gamma_G\circ\pi_2}}d(x,y)
  \end{align*}
\end{proof}

Under stronger distributivity conditions, this can be extended to infinite products.
\begin{propositionrep}\label{prop:duality_prod_infinite}
  Let $\corresp{F_i}{"\ev"_{F_i}}{\Gamma_{F_i}}$ be "correspondences". If one of the
  following holds:
  \begin{enumerate}
  \item $F_i$ has "finite couplings" and $"\Vv"$ is meet-infinite
    distributive, meaning that for all $x\in"\Vv"$ and $V\subseteq"\Vv"$,
    $x\vee\bigwedge V=\bigwedge\{x\vee v\mid v\in V\}$; or,
  \item $"\Vv"$ is completely distributive, meaning that for all sets
    $K\subseteq I\times J$ such that $K$ projects onto $I$, and any
    subset $\{x_{ij}\mid (i,j)\in K\}\subseteq"\Vv"$,
    $\bigwedge_{i\in I}\left(\bigvee_{j\in K(i)}x_{ij}\right)=
    \bigvee_{f\in A}\left(\bigwedge_{i\in I}x_{if(i)}\right)$ where
    $A=\{f\colon I\to J\mid \forall i\in I,~f(i)\in K(i)\}$;
  \end{enumerate}
  then we have a "correspondence"
  $\corresp{\prod
  F_i}{\bigwedge("\ev"_{F_i}\circ\pi_i)}{\bigcup(\Gamma_{F_i}\circ\pi_i)}$.
\end{propositionrep}
\begin{proof}
  This is the same proof as for Proposition~\ref{prop:duality_prod}
  but with the right notion of distributivity and an infinite number
  of functors instead of just two.
\end{proof}

\begin{example}
  The "quantales" $\2$ and $[0,M]$ from
  Example~\ref{ex:quantales} are both completely distributive.
\end{example}

Replacing the infinite distributivity conditions on the "quantale" by
countably infinite versions directly gives similar
statements for countable products.

In order for the inverse process from products to components to be
uniquely defined we first need to ensure that the functor has been
decomposed as much as possible using coproducts following
Propositions~\ref{prop:duality_coprod},~\ref{prop:duality_from_coprod},
and~\ref{prop:duality_coprod_equiv}. The following well-known lemma is
useful for this (see~\cite[Proposition
1.3.12]{pare2024tautfunctorsdifferenceoperator}):
\begin{lemmarep}
  Let $F \colon \Set \rightarrow \Set$ be a functor. There is a family of functors $\{F_i\}_{i\in F\{\top\}}$ indexed
  by $F\{\top\}$ such that $F\simeq\coprod F_i$, and if
  $F_i=\coprod G_j$ then all $G_j$ but one are the empty functor sending any
  set to the empty set.
\end{lemmarep}
\begin{proof}
  As $\{\top\}$ is terminal in $\Set$, for every set $X$ there is a
  unique map $\inc\colon X\to\{\top\}$. Given $i\in F\{\top\}$ we
  define $F_iX=(F\inc)^{-1}(i)$ which is a subset of $FX$. On maps we
  define $F_if$ as the restriction of $Ff$ to the sets induced by
  $F_i$. Clearly $F_i$ maps identities to identities and compositions
  to compositions, proving $F_i$ is a functor.

  Remains to prove that $F\simeq\coprod F_i$. As $F\inc$ is a
  well-defined function, sets $F_iX$ form a partition of $FX$. Given a
  map $f\colon X\to Y$ that we can now write
  $Ff\colon \uplus F_iX\to\uplus F_iY$. we need only prove that $Ff$
  sends $F_iX$ to $F_iY$. This is given by finality of $\{\top\}$,
  giving $\inc_X=f\circ\inc_Y$ and then applying $F$ and using its
  functoriality.

  For the second part of the statement, we use a kind of converse: if 
  a functor $F$ can be written as $F\simeq\coprod F_i$ and if $F\{\top\}$ is 
  a singleton, then all but one of $F_i\{\top\}$ must be empty.
\end{proof}

Having characterised how functors decompose along coproducts and how
"correspondences" work with regard to such decompositions, we can now
assume functors not to be writable as non-trivial coproducts meaning
that if a functor can be written as a coproduct, all but one of the
components must be the empty functor. Furthermore we assume for
convenience that functors are not empty functors. Those two hypotheses
can be formalised as follows: by the previous lemma, functors send $\{\top\}$
and by extension any singleton to a singleton.
Hence, given $"\ev"$ "well-behaved", with
$i\colon\{\top\}\hookrightarrow"\Vv"$,
$Fi(F\{\top\})="\ev"^{-1}(\top)$ must be a singleton.

\begin{propositionrep}\label{prop:duality_from_products}
  Given a "correspondence" $\corresp{\prod_{i\in I}F_i}{"\ev"}{\Gamma}$, we again have
  "correspondences" $\corresp{F_i}{"\ev"_{|i}}{\Gamma_{|i}}$ where
  $
    "\ev"_{|i}(x)="\ev"(\top_1,\dots,\top_{i-1},x,\top_{i+1},\dots,\top_I)
  $,
  with $(\top_j)_j="\ev"^{-1}(\top)$ and similarly for "modalities"
  in $\Gamma_{|i}$.
\end{propositionrep}
\begin{proof}
  Unfold definitions for the wanted "correspondences" and apply the
  "correspondence" in hypothesis.
\end{proof}

In a similar fashion as for the coproduct, we would like that going
back and forth between products and their components, following
Proposition~\ref{prop:duality_prod}
or~\ref{prop:duality_prod_infinite} and
Proposition~\ref{prop:duality_from_products}, preserves
"correspondences". We can only do that partially:
\begin{propositionrep}
  Whenever the "correspondences" on the left exist and we have the
  right distributivity conditions from
  Proposition~\ref{prop:duality_prod}
  or~~\ref{prop:duality_prod_infinite} for the "correspondence" on the
  right to exist, we have the following "equivalence" and
  "inequality":
  \begin{align*}
    \corresp{F_i}{"\ev_i"}{\Gamma_i}\quad&"\sim"\quad\corresp{F_i}{\left(\bigwedge("\ev"_{j}\circ\pi_j)\right)_{|i}}
    {\left(\bigcup(\Gamma_{j}\circ\pi_j)\right)_{|i}}\,\\
    \corresp{\prod F_i}{"\ev"}{\Gamma}\quad&"\le"\quad\corresp{\prod F_i}{\bigwedge("\ev"_{|i}\circ\pi_i)}
    {\bigcup({\Gamma_{|i}}\circ\pi_i)}\,
  \end{align*}
\end{propositionrep}
\begin{proof}
  For the first "equivalence" the third condition in the definition of "well-behaved" "modalities"
  gives $\left(\bigwedge("\ev"_{F_j}\circ\pi_j)\right)_{|i}="\ev_i"$ proving that
  the "coupling-based liftings" on both side coincide.

	For the
  inequality we have the bijection
  $\colon"\Vv"\simeq"\Vv"\times\{\top\}$ and inclusion
  $i\colon "\Vv"\times\{\top\}\in\le$ in the following commutative
  diagram:
  \[
    \begin{tikzcd}[ampersand replacement=\&]
      \kl[quantale]{\Vv}\arrow{r}{\id\times\top}
      \&\kl[quantale]{\Vv}\times\{\top\}\arrow{d}{i}\arrow{dr}{\pi_2}
      \&\\
      \arrow[leftarrow]{u}{\id}
      \kl[quantale]{\Vv}\arrow[leftarrow]{ur}{\pi_1}
      \arrow[leftarrow]{r}{\pi_1}
      \&\le\arrow{r}{\pi_2}\&\{\top\}
    \end{tikzcd}
  \]
  Applying a functor $F$ to the diagram preserves injections, bijections, as well as
  identities proving that we can choose $v\in F"\Vv"$ and find
  $c\in F\le$ that projects on $v$ on one side and on the only element
  of $F\{\top\}$ on the other side, using $F\pi_1$ and $F\pi_2$
  respectively. That shows, by a characterisation of the monotonicity
  of $\tau$
  using the standard canonical relation lifting of $\le$ on $"\Vv"$ to
  a relation on $\prod F_i"\Vv"$ (see, e.g.,~\cite[Section 5.1]{bkp:up-to-behavioural-metrics-fibrations-journal}) that $(\top_i)_i$ as defined in
  Proposition~\ref{prop:duality_from_products} is maximal, and that
  every $\top_i$ is maximal for the lifting to $F_i"\Vv"$. Using
  this, by monotonicity of $\tau$ and definition of $\tau_{|i}$ we get
  $\tau\le\tau_{|i}\circ\pi_i$ and then
  \[
    \tau\le\bigwedge(\tau_{|i}\circ\pi_i)
  \]
  giving the wanted inequality for the "coupling-based liftings" and by
  extension for the "correspondences" in the statement.
\end{proof}
The "inequality" comes from the following aspect of the proof. When using
the "correspondences" on the components of a product of functors
to retrieve a "correspondence" on the product itself, a choice
is made to build "modalities" for the product. The natural choice
of using a meet ensures the "inequality" while there is no obvious way to
get an "equivalence" of "correspondences".

\subsection{\kl[constant]{Constant} and "identity functors"}
\label{sec:correspondence_basic}

In this subsection we give "correspondence" results for "constant functors",
and for the "identity functor". For "constant functors", the essence
is given by the constant-to-$\cate{1}$ functor.

\begin{proposition}\label{prop:duality_constant_1}
  The triple $\corresp{\cate{1}}{"\ev"}{\Gamma}$ with the functor
  $\cate{1}$ sending any set to a fixed singleton is a
  "correspondence" if and only if $\Gamma=\{"\ev"\}$ and $"\ev"$ is the
  constant to $\top$ "modality".
\end{proposition}
\begin{proof}
  We write $\cate{1}=\{*\}$.

  $\Rightarrow$: as "modalities" $"\tau"\colon\{*\}\to"\Vv"$ are
  assumed "well-behaved", given the inclusion
  $i\colon\{\top\}\hookrightarrow"\Vv"$ we must have
  $\cate{1}i[\cate{1}\{\top\}]="\ev"^{-1}(\top)$ which translates to
  $"\ev"^{-1}(\top)=\{*\}$.

  $\Leftarrow$: both "liftings" lift all "pseudometrics" to the one
  constant to $\top$.
\end{proof}
Next, we extend this result to arbitrary constant functors. There is only one correspondence result in this case.
\begin{corollary}\label{prop:duality_constant}
  \AP For $A$ a set, denoting by $"A"$ the associated ""constant functor""
  mapping all sets to $A$, there is a "correspondence"
  $\corresp{A}{"\ev"_\top}{\{"\ev_a" \mid a\in A\}}$ with $"\ev"_\top$ the
  constant to $\top$ "modality" and
  \[
    "\ev_a"(b)=
    \begin{cases}
      \top\text{ when }b=a\\
      \bot\text{ otherwise}
    \end{cases}
  \]
  Furthermore, any other "correspondence" $\corresp{A}{"\ev"}{\Gamma}$
  is "equivalent" to $\corresp{A}{"\ev"_\top}{\{"\ev_a" \mid a\in A\}}$.
\end{corollary}
\begin{proof}
  This is a direct consequence of
  Propositions~\ref{prop:duality_coprod},~\ref{prop:duality_from_coprod},~\ref{prop:duality_coprod_equiv}
  and~\ref{prop:duality_constant_1} viewing the functor constant and
  equal to $A$ as a coproduct:
  \[
    A\simeq\coprod_{a\in A}\cate{1} \qedhere
  \]
\end{proof}

\AP We give a general "duality" result for the ""identity
functor"" $"\Id"$ acting as the identity on both sets and functions,
generalising~\cite[Example 5.29]{bbkk:coalgebraic-behavioral-metrics}.
An instance of "well-behaved" "modalities" for the "identity functor" was discussed in Example~\ref{ex:well-behaved}.
\begin{proposition}\label{prop:duality_identity}
  For any "well-behaved" "modality" $"\ev"\colon"\Vv"\to"\Vv"$ for the "identity functor",
  there is a "correspondence" $\corresp{"\Id"}{"\ev"}{\{"\ev"\}}$.
\end{proposition}
\begin{proof}
  By Proposition~\ref{prop:weak_duality} we need only prove
  \[
    \kl[coupling-based]{{"\Id"}^\downarrow_{"\ev"}}
    \ge
    \kl[codensity]{{"\Id"}^\uparrow_{"\ev"}}
  \]

  Let $X$ be a set, $d\colon X\times X\to"\Vv"$ a "$\Vv$-pseudometric",
  and $x,y\in X$. There is exactly one "coupling" of $x$ with $y$
  given by $(x,y)$. In particular it is \kl[optimal
  coupling]{optimal}:
  \begin{align*}
    \kl[coupling-based]{{"\Id"}^\downarrow_{"\ev"}}d(x,y)="\ev"(d(x,y))
  \end{align*}
  Furthermore,
  \begin{align*}
    \kl[codensity]{{"\Id"}^\uparrow_{"\ev"}}d(x,y)=\bigwedge_{f\colon d\to"\de"}
      "\de"("\ev"\circ f(x),"\ev"\circ f(y))
  \end{align*}

  Using Lemma~\ref{lem:extension_of_V-Rel_morphisms}, we
  define $f\colon d_{|\left\{ x\right\}}\to"\de"$ by $f(x)=\top$ and
  extend it to a "morphism" $f\colon d\to"\de"$. This gives $f(y)=d(x,y)$ so
  that $"\de"("\ev"\circ f(x),"\ev"\circ f(y))="\ev"(d(x,y))$. Hence
  \begin{align*}
    \kl[codensity]{{"\Id"}^\uparrow_{"\ev"}}d(x,y)\le
    "\ev"\circ d(x,y)=\kl[coupling-based]{{"\Id"}^\downarrow_{"\ev"}}d(x,y)
  \end{align*}
  ending the proof. Note that $f$ is \textit{a fortiori} an "optimal function".
\end{proof}

\subsection{Putting it together: "correspondences" for polynomial functors}
\label{sec:polynomial}

\AP Having seen "correspondences" for products, coproducts, the "identity functor",
and "constant functors", we combine these results to obtain a ``grammar''
of functors with "correspondences".

We start without products.
Using Propositions~\ref{prop:duality_coprod}
and~\ref{prop:duality_identity} as well as
Corollary~\ref{prop:duality_constant} gives "correspondences" for functors
in the following grammar of $\Set$ endofunctors:
\[
  F::="A"~\mid~"\Id"_{"\tau"}~\mid~\coprod F_i
\]
where $"\tau"$ indicates choices of "well-behaved" "modalities" in the
case of the "identity functor". Furthermore by
Propositions~\ref{prop:duality_coprod_equiv}
and~\ref{prop:duality_identity} and
Corollary~\ref{prop:duality_constant}, any "coupling-based" "lifting"
of a functor in this grammar is "equivalent" to a "correspondence"
given by our construction.

Observing that for a set $A$, $\coprod_{a\in A}F\simeq "A"\times F$, the grammar can equivalently
be defined as
\[
  F ::= "A"~\mid~"\Id"_{"\tau"}~\mid~"A"\times F~\mid~\coprod F_i
\]
where $"\tau"$ is a "well-behaved" "modality" for the "identity functor".
We can also note that these functors are exactly those isomorphic to
$"A"+"B"\times"\Id"$ for some sets $A$ and $B$.

\begin{example}[Discounting on streams]
	Consider the stream functor
  $X\mapsto A\times X$ for some alphabet $A$. The associated final
  coalgebra is the set of streams $A^{\omega}$. The results above tell
  us that to get a "correspondence" we need one "well-behaved"
  "modality" $"\ev"\colon"\Vv"\to"\Vv"$ per letter $a\in A$ for the
  "codensity lifting", through the isomorphism
  $"A"\times"\Id"\simeq\coprod_{a\in A}"\Id"$: we have a way of
  computing distance between streams in a specific way for each letter
  in $A$. When $"\Vv"=([0,M],\ge,+)$ for some real number $M\in(0,\infty]$,
  choosing $"\ev_a"(x)=c\cdot x$, the constant $c$ is then a discount
  factor allowing one to give more value to the first letters of a
  stream. We can choose a different constant $c_a$ for each letter $a$
  giving different values to different letters. We can also go
  further: if we do not want a linear discount we can use arbitrary
  "well-behaved" "modalities" for the "identity functor". On the
  "quantale" $([0,M],\ge,+)$, those are exactly sub-additive monotone
  maps mapping $0$ to $0$.
\end{example}

\AP We also get "correspondences" for the following grammar of
\emph{simple polynomial functors}:
\[
  F ::= "A"~\mid~"\Id"_{"\tau"}~\mid~"A"\times F~\mid~\coprod F_i~\mid~\prod F_i
\]
These functors all have "finite couplings". Following
Propositions~\ref{prop:duality_prod}
and~\ref{prop:duality_prod_infinite} we can either get finite or
arbitrary products in the grammar above by assuming $"\Vv"$ to be
distributive or meet-infinite distributive respectively.

\begin{example}
  \label{ex:DFA}
  We retrieve the "correspondence" of Section~\ref{sec:overview} as
  the particular case of the "quantale" $([0,M],\ge,+)$, the functor
  $2\times "\Id"^A$, and indexed "modalities" $"\tau"$ for the
  "identity functors" always mapping $x\in[0,M]$ to $c\cdot x$ for
  some $c\in[0,1)$. Replacing $[0,M]$ by the usual Boolean quantale
  $\2=\{\top,\bot\}$ and indexed "modalities" by the identity
  functions gives a "correspondence" "equivalent" to the usual
  relation lifting~\cite{DBLP:books/cu/J2016} associated to behavioural equivalence of DFA.
\end{example}

\section{"Dualities" for the "powerset" and "probability distribution functors"}
\label{sec:powerset_distribution}

We give "duality" results for the "powerset functor" mapping a set $X$
to the set of its subsets $"\Pp"X$, recovering a result from the
literature
\cite{bbkk:coalgebraic-behavioral-metrics,DBLP:journals/acs/HofmannN20}
reproduced here as Corollary~\ref{prop:duality_powerset_meet}, and
complementing it with a duality result where $\Vv$ is a total order
(Corollary~\ref{cor:p-total}).

\begin{remark}
  The non-empty powerset functor $\Pp_{\ge 1}$ mapping a set $X$ to
  the set of its non-empty subsets might have been chosen instead
  of $"\Pp"$. Note that $\Pp=\cate{1}+\Pp_{\ge 1}$. Up-to "equivalence",
  as a consequence
  of Propositions~\ref{prop:duality_coprod},
  \ref{prop:duality_from_coprod} and~\ref{prop:duality_constant_1} a
  "correspondence" for $"\Pp"$ is exactly given by one for
  $\Pp_{\ge 1}$ and conversely.
\end{remark}

\AP Given $T_1,T_2\in"\Pp"X$, $t_1\in T_1$, and a "$\Vv$-pseudometric"
$d\colon X\times X\to"\Vv"$, we write $"M_{t_1}"$ for the ""set of
maximal elements"" of $\{d(t_1,t_2)\mid t_2\in T_2\}$, and similarly
for elements $t_2\in T_2$. Expressing a "coupling-based lifting" as a
"codensity" one in a "correspondence" implies turning a join in
a meet. The following proposition gives a condition allowing one to
do that in the case of the "powerset functor" by making a "modality"
absorb a meet.
\begin{propositionrep}\label{prop:duality_powerset}
  Let $"\ev"\colon"\Pp" "\Vv"\to"\Vv"$ be a "well-behaved" "modality".
  If for every "$\Vv$-pseudometric" $d\colon X\times X\to"\Vv"$ and
  sets $T_1,T_2\in"\Pp"X$ the following equality holds
  \[
    "\ev"
    \left\{
      \bigvee "M_{t_1}"\mid t_1\in T_1
    \right\}
    \bigwedge
    "\ev"
    \left\{
      \bigvee "M_{t_2}"\mid t_2\in T_2
    \right\}
    =
    "\ev"\left(\left(
        \bigcup_{t_1\in T_1}"M_{t_1}"
      \right)
      \cup
      \left(
        \bigcup_{t_2\in T_2}"M_{t_2}"
      \right)
    \right)
  \]
  then there is a "duality" $\duality{"\Pp"}{"\ev"}{\{"\ev"\}}$.
\end{propositionrep}
\begin{proof}
  Thanks to Proposition~\ref{prop:weak_duality} we only need to prove
  that under the conditions of the statement we have:
  \[
    \kl[coupling-based]{{"\Pp"}^\downarrow_{"\ev"}}\ge\kl[codensity]{{"\Pp"}^\uparrow_{"\ev"}}
  \]

  Let $X$ be a set, $d\colon X\times X\to"\Vv"$ a "$\Vv$-pseudometric",
  $T_1,T_2\in"\Pp"(X)$.

  "Couplings" $T$ of $T_1$ and $T_2$ are exactly sets made of "couplings" of
  elements of $T_1$ with elements of $T_2$ such that for all
  $t\in T_1$ (resp. $t\in T_2$) there is a "coupling" of $t$ with some
  $t'\in T_2$ (resp. $t'\in T_1$) in $T$.

  There are two possibilities; either there are no "couplings" of
  $T_1$ with $T_2$, either there is at least one.

  When there are no "couplings" of $T_1$ with $T_2$, it can be for two
  reasons: either $T_1$ or $T_2$ is empty but not the other, or
  for some $t\in T_1$ (or $t\in T_2$), for all $t'\in T_2$ (resp.
  $t'\in T_1$) there are no "couplings" of $t$ with $t'$. This last case
  is not possible as $(t,t')$ is the one and only "coupling" of
  $t\in T_1$ with $t'\in T_2$.

  When, for example, $T_1$ is empty,
  \begin{align*}
    \kl[coupling-based]{{"\Pp"}^\downarrow_{"\ev"}}d(T_1,T_2)&= \bot
  \end{align*}
  With $f\colon d\to"\de"$ constant to $\bot$,
  $"\ev"\circ "\Pp"f(T_1)="\ev"\circ "\Pp" f(\emptyset)=\top$ (by the third condition of
  "well-behaved" "modality") and $"\ev"\circ "\Pp"f(T_2)="\ev"\circ "\Pp"f(\{\bot\})=\bot$ (by
  "monotonicity" of $"\ev"$) so that
  $"\de"("\ev"\circ "\Pp"f(T_1),"\ev"\circ "\Pp"f(T_2))=\bot$. Thus
  we also get
  $\kl[codensity]{{"\Pp"}^\uparrow_{"\ev"}}d(T_1,T_2)=\bot$.

  We now assume that neither $T_1$ nor $T_2$ is empty. We have:
  \begin{align*}
    \kl[coupling-based]{{"\Pp"}^\downarrow_{"\ev"}}d(T_1,T_2)
    &=\bigvee_{T\in"\Omega"(T_1,T_2)}"\ev"\circ"\Pp"d(T)
  \end{align*}
  Let us see what we can say of $"\ev"$ following it being "monotone":
  for every pair of predicates $p,q\colon X\to"\Vv"$, whenever
  $p\le q$ we have $"\ev"\circ "\Pp" p\le"\ev"\circ"\Pp" q$. This is
  equivalent to stating that whenever $A,B\subseteq"\Vv"$ are such
  that every maximal element of $A$ is below some maximal element of
  $B$ and every minimal element of $B$ is above some minimal element
  of $A$ then we have $"\ev"(A)\le"\ev"(B)$. Hence $"\ev"$ only depend
  on the maximal and minimal elements of its arguments.

  As a consequence of the conditions in the statement we will see that
  the following "coupling" is \kl[optimal coupling]{optimal}:
  \[
    T=
    \left( \bigcup_{t_1\in T_1}\{(t_1,t_2)\mid d(t_1,t_2)\text{ is maximal}\} \right)
    \bigcup
    \left( \bigcup_{t_1\in T_1}\{(t_1,t_2)\mid d(t_1,t_2)\text{ is maximal}\} \right)
  \]
  We have
  \begin{equation}
    \label{eq:opt-coup}
    "\ev"\circ"\Pp"d(T)= "\ev"\left(\left( \bigcup_{t_1\in
          T_1}"M_{t_1}" \right) \cup \left( \bigcup_{t_2\in
          T_2}"M_{t_2}" \right) \right)\le\kl[coupling-based]{{"\Pp"}^\downarrow_{"\ev"}}d(T_1,T_2)
  \end{equation}
  We prove that the "codensity lifting" reaches this value, which will
  prove the wanted inequality, by considering the right map
  $f\colon d\to"\de"$.

  Consider $f\colon d\to"\de"$ defined by $\top$ on $T_2$ and extended
  on $T_1$ using Lemma~\ref{lem:extension_of_V-Rel_morphisms}:
  $f(x)=\bigvee_{t_2\in T_2}d(x,t_2)$.

  We get the following:
  \begin{align*}
    "\de"("\ev"\circ"\Pp"f(T_1),"\ev"\circ"\Pp" f(T_2))
    &= "\de"\left( "\ev"\left\{ \bigvee_{t_2\in T_2}d(t_1,t_2)\mid t_1\in T_1 \right\},
      "\ev"\left\{ \top \right\} \right)\\
    &= "\ev"\left\{ \bigvee_{t_2\in T_2}d(t_1,t_2)\mid t_1\in T_1 \right\}\\
    &= "\ev"\left\{\bigvee "M_{t_1}"\mid t_1\in T_1\right\}
  \end{align*}
  By symmetry we get
  \[
    \kl[codensity]{{"\Pp"}^\uparrow_{"\ev"}}d(T_1,T_2)\le
    "\ev"
    \left\{
      \bigvee "M_{t_1}"\mid t_1\in T_1
    \right\}
    \bigwedge
    "\ev"
    \left\{
      \bigvee "M_{t_2}"\mid t_2\in T_2
    \right\}
  \]
  Combining the hypothesis and Equation~\ref{eq:opt-coup} gives the final result.
\end{proof}

\begin{corollaryrep}\label{cor:p-total}
  Let $"\ev"\colon"\Pp" "\Vv"\to"\Vv"$ be a "well-behaved" "modality".
  If the order on $"\Vv"$ is total then there is a "duality"
  $\duality{"\Pp"}{"\ev"}{\{"\ev"\}}$.
\end{corollaryrep}
\begin{proof}
  In this case each set $"M_t"$ is reduced to a singleton and maximal
  elements in the three sets
  $\left( \bigcup_{t_1\in T_1}"M_{t_1}" \right) \cup \left(
    \bigcup_{t_2\in T_2}"M_{t_2}" \right)$,
  $ \left\{ \bigvee "M_{t_1}"\mid t_1\in T_1 \right\}$, and
  $ \left\{ \bigvee "M_{t_2}"\mid t_2\in T_2 \right\}$ are actually
  maximum that are all $\max_{t_i\in T_i} d(t_1,t_2)$.

  The minimum that appears in
  $\left( \bigcup_{t_1\in T_1}"M_{t_1}" \right) \cup \left(
    \bigcup_{t_2\in T_2}"M_{t_2}" \right)$ must appear in one of the
  other two sets and as $"\tau"$ being "monotone" implies that it
  depends only on maximal and minimal elements (see the proof of
  Proposition~\ref{prop:duality_powerset}) we get the necessary equality
  \[
    "\ev"
    \left\{
      \bigvee "M_{t_1}"\mid t_1\in T_1
    \right\}
    \bigwedge
    "\ev"
    \left\{
      \bigvee "M_{t_2}"\mid t_2\in T_2
    \right\}
    =
    "\ev"\left(\left(
        \bigcup_{t_1\in T_1}"M_{t_1}"
      \right)
      \cup
      \left(
        \bigcup_{t_2\in T_2}"M_{t_2}"
      \right)
    \right)
  \]
  to apply Proposition~\ref{prop:duality_powerset}.
\end{proof}

\begin{corollaryrep}\label{prop:duality_powerset_meet}
  We assume $"\Vv"$ to be completely distributive.
  We have "duality" for the "coupling-based" and "codensity liftings"
  of the "powerset functor", both along the meet "modality" which maps
  a subset of $"\Vv"$ to the meet of its elements.
\end{corollaryrep}
\begin{proof}
  Being completely distributive with the meet "modality" directly give
  the condition of Proposition~\ref{prop:duality_powerset}.
\end{proof}

For the "probability distribution functor" we fix a constant
$M\in(0,\infty]$ and use the "quantale" $([0,M],\ge,+)$. The
following (Kantorovich-Rubinstein) "duality" result is well-known (see, e.g.,~\cite[Theorem
5.10]{villani2009optimal} for a general proof in the continuous case
and~\cite[Appendix A]{jacobs2024drawingdistance} for the discrete
finite case):
\begin{proposition}\label{prop:duality_distr}
  \AP There is a "duality" $\duality{"\Dd"}{"\E"}{\{"\E"\}}$ with $"\Dd"$
  the "finite probability distribution functor" and $"\E"$ giving the
  "expectation" of probability distributions.
\end{proposition}
We extend this result slightly to allow post-composition by
"well-behaved" "modalities":

\begin{propositionrep}\label{prop:duality_dd_extended}
  Let $"\ev"\colon[0,M]\to[0,M]$ be a "well-behaved"
  "modality". If $"\ev"$ is additive, meaning that for all $x,y\in[0,M]$,
  $\tau(x+y)=\tau(x)+\tau(y)$, then there is
  a "duality" $\duality{"\Dd"}{"\ev"\circ"\E"}{\{"\ev"\circ"\E"\}}$.
\end{propositionrep}
\begin{proof}
  It is known that infima and suprema in the "duality" of
  Proposition~\ref{prop:duality_distr} are actually minima and
  maxima: there are both "optimal couplings" for the "coupling-based lifting" and "optimal functions" for the "codensity lifting" of
  $\duality{"\Dd"}{"\E"}{\{"\E"\}}$. Hence the "duality" can be
  written: for all probability distributions $P_1,P_2$ on sets $X$:
  \[
    "\E"_Pd=
    \left|
      "\E"_{P_1}f-"\E"_{P_2}f
    \right|
  \]
  for some "optimal coupling" $P$ of $P_1$ and $P_2$ and some "optimal non-expansive function" $f\colon d\to|\_-\_|$. Post-composing by
  $"\ev"$ obviously gives the expected result.
\end{proof}

\section{"Correspondences" for grammars of functors}
\label{sec:grammar}

\AP In this section we combine the results from the previous two sections, allowing
the construction of "correspondences" for certain classes of functors including "powerset", "distribution",
"identity" and "constant functors", as well as products and coproducts thereof.

Whenever $"\Vv"$ is totally ordered and completely distributive, using
results from Sections~\ref{sec:correspondence_abstract}
and~\ref{sec:powerset_distribution} we can construct "correspondences"
for the following grammar of functors:
\[
  F ::= "A"~\mid~"\Id"_{"\tau"}~\mid~"\Pp"_{"\tau"}~\mid~\prod F_i~\mid~\coprod F_i
\]
where the indices $"\tau"$ are "well-behaved" "modalities" for either
the "identity" or "powerset functors".

Noting that
$"\Pp"\left( \coprod_{i\in I} F_i\right)\cong \prod_{i\in
  I}"\Pp"(F_i)$, this grammar can be reformulated as follows:
\begin{align*}
  G ::=&~"A"~\mid~"\Id"_{"\tau"}~\mid~"A"\times G~\mid~\coprod G\\
  F ::=&~"A"~\mid~"\Id"_{"\tau"}~\mid~"\Pp"\circ G~\mid~\prod F_i~\mid~\coprod F_i
\end{align*}
where the indices of the form $"\tau"$ indicate a choice of possible
"well-behaved" "modalities" for the "identity functor" when appearing
in $F$ and for the "powerset functor" when appearing in $G$.

\begin{example}
  In a similar fashion as in Example~\ref{ex:DFA}, considering the
  completely distributive "quantale" $([0,M],\ge,+)$, the functor
  $2\times"\Pp"^A$ which is associated to \emph{non-deterministic
    automata}, and indexed "modalities" for the "powerset functors" to
  be all mapping finite subsets $V\subset[0,M]$ to $c\cdot \max(V)$
  for a fixed $c\in [0,1)$, we get a "lifting" associated to a
  "shortest-distinguishing-word-distance" for non-deterministic
  systems as is detailed in Section~\ref{sec:overview} for DFA.
\end{example}

\begin{remark}
  This grammar with finite products only defines a fragment of finitary Kripke polynomial
  functors as defined in~\cite{DBLP:books/cu/J2016} which correspond to replacing $"\Pp"(G)$ by $"\Pp"(F)$, and having
  only finite products but including arbitrary exponents.
\end{remark}

\begin{remark}
  Using the identity "modality" for the "identity functors" and
  the meet "modality" for the "powerset functors", the "modalities"
  given by this construction for the "coupling-based" "liftings" are exactly
  the \emph{canonical evaluation maps} as defined in~\cite{bkp:up-to-behavioural-metrics-fibrations-journal}.
\end{remark}

In the particular case of $"\Vv"=[0,M]$ we can extend the
above grammar with the "finite probability distribution functor"
and get "correspondence" results for the following:
\[
  F::="A"~\mid~"\Id"_{"\tau"}~\mid~"\Pp"_{"\tau"}~\mid~"\Dd"_{"\tau"}~\mid~\prod F_i~\mid~\coprod F_i
\]
where the indices $"\tau"$ indicate choices of "well-behaved" "modalities" for
either the "identity", "powerset", or "distribution functors",
furthermore of a "modality" of the form $"\tau"\circ"\E"$ for the latter as described in
Proposition~\ref{prop:duality_dd_extended}.

Observe that
$"\Dd"\left(\coprod_{i\in I}F_i\right)\cong"\Dd"(I)\times\prod_{i\in
  I}"\Dd"(F_i)$ so that the grammar above can be expressed as follows:
\begin{align*}
  G::=&~"A"~\mid~"\Id"_{"\tau"}~\mid~"A"\times G~\mid~\coprod G\\
  F::=&~"A"~\mid~"\Id"_{"\tau"}~\mid~"\Pp"\circ G~\mid"\Dd"\circ G~\mid~\prod F_i~\mid~\coprod F_i
\end{align*}
where the indices $"\tau"$ indicate choices of possible
"well-behaved" "modalities" for the "identity functor" in $"\Id"_{"\tau"}$
and for the "powerset" or "distribution functors" when appearing in $G$ for
$"\Pp"\circ G$ and $"\Dd"\circ G$ respectively.

\begin{example}
  We fix $M=1$ so that $"\Vv"=[0,1]$. Consider the functor $"\Dd"(\cate{1}+"\Id")^A$ for a set
  of labels $A$ and $\cate{1}$ a singleton as in
  Proposition~\ref{prop:duality_constant_1}. It is associated to \emph{labelled
  Markov processes}. To get a "correspondence" we need one "modality"
  $"\tau_a"\colon"\Dd" [0,1]\to[0,1]$ per label $a\in A$. By
  Proposition~\ref{prop:duality_dd_extended} we can take them all to
  be $"\tau_a"(\mu)=c\cdot"\E"(\mu)$ with $c\in[0,1)$ a fixed constant.
  The resulting "correspondence" is associated to the usual
  metrics for labelled Markov processes as introduced in~\cite{dgjp:metrics-labelled-markov}. This 
  is a direct consequence of~\cite[Proposition 29]{bw:behavioural-pseudometric}
  which expresses the associated "lifting" for one label
  in a ``codensity-like'' manner.
\end{example}

\section{A counter-example: "conditional transition systems"}
\label{sec:cts}

We now highlight an example where a "correspondence" can \emph{not} be provided by our construction:
some "codensity lifting" may not be obtained as a "coupling-based lifting".

\AP A ""conditional transition system"" ("CTS") over an alphabet $A$
and a partially ordered set of conditions $(\Ll,\le)$ is a tuple
$(X,A,\Ll,\delta)$ where $X$ is a set of states and the transition map
$\delta\colon X\times A \to ((\Ll,\le)\to("\Pp"(X),\supseteq))$
associates to each pair $(x,a)\in X\times A$ a monotone function
$\delta_{x,a}$ from the poset of conditions to $"\Pp"(X)$.
For the associated functor to live in $\Set$ we restrict to the "CTSs"
for which $\Ll$ is trivially ordered by $\forall x,y\in\Ll, x\le y$.
Hence any map $\delta_{x,a}\colon\Ll\to"\Pp"(X)$ is monotone, and "CTSs" are
exactly coalgebras for the functor $F="\Pp"(-)^{A\times\Ll}$ (see
\cite{DBLP:journals/lmcs/Beohar0k0W18}).

\AP The associated bisimulations, not detailed here, are related to a
"lifting"
$""\overline{F}""\colon"\Pp(\Ll)\cate{-PMet}"\to"\Pp(\Ll)\cate{-PMet}"$
of $F$ defined by
\begin{align*}
  "\overline{F}"r(T_1,T_2)=\left\{l\in\Ll\mid\right.
  \forall a\in A,&\forall (x,y)\in T_1(l,a)\times T_2(l,a),\\
  &\exists (x',y')\in T_1(l,a)\times T_2(l,a),\\
  &~~\left.(l\in r(x,y'))\text{ and }(l\in r(x',y))\right\}
\end{align*}
on objects and $"\overline{F}"f=Ff$ on maps.

\begin{propositionrep}
  For $a\in A$ consider a "modality" $"\ev_a"\colon "\Pp"("\Pp"(\Ll))^{\Ll\times A}\to"\Pp"(\Ll)$, defined by :
  \[
    \forall f \colon \Ll\times A\to"\Pp"("\Pp"(\Ll)),  "\ev_a"(f)=\left\{ l\in\Ll\mid l\in \cap f(l,a) \right\}
  \]
  The "codensity lifting" defined by the family $("\ev_a")_{a\in A}$
  of "modalities" is equal to $"\overline{F}"$.
\end{propositionrep}
\begin{proof}
  Take $d\colon X\times X\to\Pp(\Ll)$ and
  $T_1,T_2\in\Pp(X)^{\Ll\times A}$. We want to prove the
  following:
  \[
    "\overline{F}"d(T_1,T_2)=\kl[codensity]{F^\uparrow_{\{"\ev_a"\}}}d(T_1,T_2)
  \]
  Explicitly we want
  \begin{align*}
    \left\{ l\in\Ll\mid\forall a\in A,\right.
      &\forall x,y\in T_1(l,a)\times T_2(l,a),\\
      &\exists x',y'\in T_1(l,a)\times T_2(l,a),\\
      &~~\left.(l\in d(x,y'))\text{ and }(l\in d(x',y))\right\}
  \end{align*}
  to be exactly the same as
  \[
    \underset{a\in A,~f\colon d\to"\de"}{\bigwedge}"\de"("\ev_a"\circ
    Ff(T_1),"\ev_a"\circ Ff(T_2))
  \]
  Suppose that for $l\in\Ll$ there exists $a\in A$ and
  $f\colon r\to"\de"$ such that
  $l\notin "\de"("\ev_a"\circ Ff(T_1),"\ev_a"\circ Ff(T_2))$. By
  definition of $"\de"$, we have
  $"\de"(X,Y)=(X\cap Y)\cup(\Ll\backslash(X\cup Y))$. Up-to to a permutation
  of $T_1$ and $T_2$ this gives $l\in"\ev_a"\circ Ff(T_1)$ and
  $l\notin"\ev_a"\circ Ff(T_2)$:
  \begin{align*}
    \begin{cases}
      l\in"\ev_a"\circ Ff(T_1)\\
      l\notin"\ev_a"\circ Ff(T_2)
    \end{cases}\Leftrightarrow&
    \begin{cases}
      l\in\left\{ c\in\Ll\mid c\in\cap Ff[T_1(c,a)] \right\}\\
      l\notin\left\{ c\in\Ll\mid c\in\cap Ff[T_2(c,a)] \right\}
    \end{cases}\\
    \Leftrightarrow&
    \begin{cases}
      l\in\cap Ff[T_1(l,a)]\\
      l\notin\cap Ff[T_2(l,a)]
    \end{cases}\\
    \Leftrightarrow&
    \begin{cases}
      \forall x\in T_1(l,a),~l\in f(x)\\
      \exists y\in T_2(l,a),~l\notin f(y)
    \end{cases}\\
    \Leftrightarrow&
    \exists y\in T_2(l,a),~\forall x\in T_1(l,a),~l\notin"\de"(fx,fy)
  \end{align*}
  As $f$ is a "morphism" of "pseudometrics", $d\le"\de"\circ(f\times f)$ and we get
  that $l\notin d(x,y)$ for all $x\in T_1(l,a)$ and some
  $y\in T_2(l,a)$.

  Conversely suppose that there exists $y\in T_2(l,a)$ such that for
  all $x\in T_1(l,a)$ we have $l\notin d(x,y)$. Consider
  $f\colon d\to"\de"$ defined by $\top$ on every elements of
  $T_1(l,a)$ and extended using
  Lemma~\ref{lem:extension_of_V-Rel_morphisms}. We know that
  $l\in f(x)$ for all $x\in T_1(l,a)$. Using the definition of $f$
  given by Lemma~\ref{lem:extension_of_V-Rel_morphisms} we know that
  \[
    f(y)=\bigcup\left\{ d(x,y)\mid x\in T_1(l,a) \right\}
  \]
  Because for all $x\in T_1(l,a)$, $l\notin d(x,y)$ we get that
  $l\notin f(y)$. Thus we retrieve the condition at the end of the
  sequence of equivalences above and we know that
  $l\notin\kl[codensity]{F^\uparrow_{\{"\ev_a"\}}} d(T_1,T_2)$.

  We have proven the following: $l\notin\kl[codensity]{F^\uparrow_{\{\ev_a\}}} d(T_1,T_2)$ if
  and only if there exits $a\in A$, $f\colon d\to"\de"$, and
  $y\in T_i(l,a)$ such that for all $x\in T_j(l,a)$ $l\notin r(x,y)$
  (where $i\neq j$ are in $\left\{ 1,2 \right\}$).

  This is exactly the condition for $l\notin "\overline{F}"r(T_1,T_2)$ so that
  $\kl[codensity]{F^\uparrow_{\{"\ev_a"\}}} r(T_1,T_2)="\overline{F}"r(T_1,T_2)$ ending the proof.
\end{proof}

\begin{proposition}
  If $A$ is non-empty and $|\Ll|\ge 2$ then there is no "well-behaved" "modality"
  ${"\ev"\colon"\Pp"("\Pp"(\Ll))^{\Ll\times A}\to"\Pp"(\Ll)}$ such that the
  resulting "coupling-based lifting" is equal to $"\overline{F}"$.
\end{proposition}
\begin{proof}
  Let $a\in A$ be some letter and $c_1,c_2\in\Ll$ be two distinct conditions.
  Consider $d\colon X\times X\to\Pp(\Ll)$ constant to $\top$, that is, $\Ll$.
  It is obviously a "$\Pp(\Ll)$-pseudometric". Finally we consider
  $T_1,T_2\in\Pp(X)^{\Ll\times A}$ such that $T_1(c_1,a)=\emptyset$ and 
  $T_1(c_2,a)=X$, and $T_2(c_1,a)=T_2(c_2,a)=X$. Directly, because
  $T_1(c_1,a)=\emptyset$ but $T_2(c_1,a)\neq\emptyset$ there are no
  "couplings" of $T_1$ and $T_2$, and for all "well-behaved"
  "modalities",
  $\kl[coupling-based]{F^\downarrow_{"\ev"}}d(T_1,T_2)=\emptyset$. On
  the other hand,
  \[
    \forall (x,y)\in T_1(c_2,a)\times T_2(c_2,a),~\exists (x',y')\in T_1(c_2,a)\times
    T_2(c_2,a),~(l\in d(x,y'))\text{ and } (l\in d(x',y))
  \]
  Hence $c_2\in"\overline{F}"d(T_1,T_2)\neq\emptyset$.
\end{proof}

Hence, "conditional transition systems" give a limitation to the
"correspondence" results provided above. Note however that the
associated functor has a "correspondence" induced by a grammar above.
The reason it is not equivalent to the "lifting" $"\overline{F}"$ is that
it considers the set $\Ll\times A$ as being a set of actions, while we
want conditions in $\Ll$ to be independent of one another and compared
separately. Indeed, conditions are fixed throughout the execution of a
"CTS" whereas actions may change at each step.

\section{Conclusions and future work}

We have studied "correspondences" between "coupling-based" and "codensity liftings",
moving from the classical Kantorovich-Rubinstein duality for distributions
to different types of endofunctors on $\Set$. In particular, we have shown that
such types of "correspondences" are closed under coproducts and
products and used that to provide explicit "correspondences" for
several grammars of functors, including polynomial functors
with the possibility of extending them using the "powerset" and the "probability
distribution functors". This instantiates
to usual "liftings" of functors associated to (non)deterministic
finite automata, or labelled Markov processes, both with discount.

In~\cite{DBLP:conf/fossacs/GoncharovHNSW23} the authors have shown
that on an abstract level all "coupling-based liftings" are in fact
"codensity liftings", implying that all "coupling-based liftings" have
some associated "correspondences". Section~\ref{sec:cts} shows that
the converse does not hold by providing an example of "lifting" that
arises as a "codensity lifting" but not as a "coupling-based" one. Our
work proves that "correspondences" for coproducts of functors are
characterised by ones for the components of the coproducts, but gives
no such result for products of functors. For example we could consider
the "coupling-based lifting" of the \emph{diagonal functor}
$\Delta\colon X\mapsto X\times X$ along the "well-behaved" "modality"
$"\otimes"\colon "\Vv"\times"\Vv"\to"\Vv"$. The question of whether it arises
from "coupling-based liftings" of
the identity functors, and the problem of relating it to a "codensity lifting" in a
"correspondence" remain open. Note however that if "coupling-based
liftings" require their "modalities" to be "well-behaved", "codensity
liftings", as defined in~\cite{DBLP:conf/fossacs/GoncharovHNSW23} need
no such assumption and can be defined along sets of any "modalities".
Hence it is possible that to see the "coupling-based lifting"
of $\Delta$ along $"\otimes"$ as a "codensity lifting" it is
needed to consider general "modalities" for the latter.

In Section~\ref{sec:cts} we have seen that a certain "lifting" for "conditional transition systems"
can not be obtained as a "coupling-based lifting". This leads to the question of whether
the definition of "coupling-based liftings" could be
extended somehow to encompass this example and obtain a "correspondence" with the existing "codensity lifting".

\bibliography{bibliography.bib}

\end{document}